\documentclass[11pt, one sided]{amsart}
\RequirePackage[colorlinks,citecolor=blue,urlcolor=blue]{hyperref}

\usepackage{latexsym,amssymb,amsmath,amsfonts,amsthm}
\usepackage{amsthm,amsmath}
\usepackage{float}
\usepackage{enumerate}
\usepackage{comment}
\usepackage[margin=3cm]{geometry}
\usepackage{bbm}
\usepackage{subfigure}
\usepackage{graphicx}
\usepackage[toc,page]{appendix}
\usepackage{multirow}
\usepackage{amsfonts}
\usepackage[all]{xy}
\usepackage{caption}
 \usepackage{geometry}                
\usepackage{graphicx}
\usepackage{amssymb}
\usepackage{epstopdf}
\usepackage{color}
\usepackage{bbm, dsfont}

\usepackage{mathabx} 

\usepackage{hyperref}
\usepackage[utf8]{inputenc}
\usepackage{amssymb,amsthm,amsmath,amssymb,wrapfig,dsfont}
\usepackage[dvipsnames]{xcolor}
\definecolor{myred}{RGB}{251,154,133}
\definecolor{myblue}{RGB}{153,206,227}
\definecolor{mylightblue}{RGB}{0, 150, 255}
\definecolor{mygreen}{RGB}{32, 210, 64}
\definecolor{mygray}{RGB}{220, 220, 220}

\usepackage{tikz}
\usetikzlibrary{decorations.pathmorphing}
\tikzset{snake it/.style={decorate, decoration=snake}}
\usetikzlibrary{shapes.geometric,positioning,decorations.pathreplacing} 
\usepackage{pgfplots}

\newtheorem{theorem}{Theorem}
\newtheorem{definition}{Definition}[section]
\newtheorem{lemma}{Lemma}
\newtheorem{remark}{Remark}[section]
\newtheorem{proposition}{Proposition}[section]
\newtheorem{corollary}{Corollary}[section]

\newtheorem{open_problem}{Open Problem}

\def\beq{ \begin{equation} }
\def\eeq{ \end{equation} }

\def\ep{\varepsilon}

\def\square{\vcenter{\vbox{\hrule height .4pt
  \hbox{\vrule width .4pt height 5pt \kern 5pt
        \vrule width .4pt} \hrule height .4pt}}}

\newcommand{\BE}{{\mathbb{E}}}

\newcommand{\BN}{{\mathbb{N}}}

\newcommand{\BP}{{\mathbb{P}}}

\newcommand{\BR}{{\mathbb{R}}}

\newcommand{\BZ}{{\mathbb{Z}}}

\newcommand{\ind}{{\mathbbm{1}}}

\newcommand{\prob}{{\bf P}}

\newcommand{\bae}{\begin{equation}\begin{aligned}}
\newcommand{\eae}{\end{aligned}\end{equation}}

\newcommand{\Z}{\mathbb{Z}}

\DeclareFontFamily{OML}{rsfs}{\skewchar\font'177}
\DeclareFontShape{OML}{rsfs}{m}{n}{ <5> <6> rsfs5 <7> <8> <9>
rsfs7 <10> <10.95> <12> <14.4> <17.28> <20.74> <24.88> rsfs10 }{}
\DeclareMathAlphabet{\mathfs}{OML}{rsfs}{m}{n}

\pgfplotsset{compat=1.18}
\newcommand{\abs}[1]{\left| #1 \right|}
\newcommand{\rb}[1]{\left( #1 \right)}
\newcommand{\bb}[1]{\left[ #1 \right]}
\newcommand{\cb}[1]{\left\{ #1 \right\}}

\newcommand{\E}[2][]{\mathbb{E}_{#1}\left[#2\right]} 
\newcommand{\Prob}[2][]{\mathbb{P}_{#1}\left(#2\right)}
\newcommand{\indicator}[1]{\mathbbm{1}_{\left\{#1\right\}}}

\begin{document}

\title{Cluster-Cluster model in $\mathbb{Z}^d$}

\author{Noam Berger}
\address[Noam Berger]{Department of mathematics, TUM}
\urladdr{https://www.math.cit.tum.de/en/probability/people/berger/}
\email{noam.berger@tum.de }

\author{Eviatar B. Procaccia}
\address[Eviatar B. Procaccia]{Faculty of data and decision sciences, Technion}
\urladdr{https://procaccia.net.technion.ac.il}
\email{eviatarp@technion.ac.il}

\author{Dominik Schmid}
\address[Dominik Schmid]{Department of mathematics, University of Augsburg}
\urladdr{https://sites.google.com/view/dominik-schmid}
\email{d.schmid@uni-a.de}

\author{Daniel Sharon}
\address[Daniel Sharon]{Faculty of data and decision sciences, Technion}
\email{danielsharon@campus.technion.ac.il}

\begin{abstract}
We consider a stochastic process on $\mathds{Z}^d$ for $d \geq 1$. Given a translation invariant and ergodic starting configuration of finite clusters,  
each cluster $\mathfs{C}$ performs a continuous time simple random walk with rate $\abs{\mathfs{C}}^{-\alpha}$. 
If it attempts to move to a vertex occupied by another cluster, it does not move, and instead the two clusters connect via a new edge.
In all dimensions, we show that if $\alpha\ge 0$, there is almost surely no spontaneous creation of an infinite cluster within finite time. Moreover, for any $\alpha\le-1-2/d$ there is a finite-time blowup almost surely. In the regime $\alpha\in(-1,0)$ we show that the behavior greatly depends on the initial configuration.
In addition, in dimension one, we establish the exact phase diagram.  
\end{abstract}

\addtocontents{toc}{\protect\setcounter{tocdepth}{1}}

\maketitle
\tableofcontents

\section{Introduction}

\subsection{Preface}
The Cluster-Cluster model was introduced by Meakin et al. in 1984 \cite{meakin1984diffusion} and studied in the physics literature \cite{meakin1984effects, meakin1985dynamic,rajesh2024exact}. The model was introduced as a more realistic model than Diffusion Limited Aggregation (DLA) for colloidal systems exhibiting flocculation, such as blood coagulation \cite{leyvraz2003scaling}, cloud formation \cite{lushnikov2006gelation}, aerosol dynamics \cite{lushnikov1978coagulation} and protein aggregation \cite{frieden2007protein}.
Intuitively, for the Cluster-Cluster model on $\Z^d$ with $d \geq 1$ and parameter $\alpha \in \BR$, we consider a subset of the vertices at time $0$. Each vertex is equipped with a rate $1$ Poisson clock. 
Whenever the clock rings, choose one of the $2d$-many unit vectors uniformly at random, and try to move the particle in this direction. The move is performed if and only if the target is vacant. 
Otherwise, we create an edge between the particle and the targeted  vertex, and we say that the respective vertices form a cluster. In general, each cluster $\mathcal{C}$ receives an independent rate $|\mathcal{C}|^{-\alpha}$ Poisson clock. Whenever the clock rings, we attempt to move the entire cluster by one of the $2d$-many unit vectors, chosen uniformly at random. The attempt is performed if and only if the target positions are vacant. Otherwise, we choose one of the sites blocking the move uniformly at random and create an edge between the cluster and the targeted  vertex. A visualization of a sample according to the Cluster-Cluster model in $d=2$ and $\alpha=1$ is given in Figure~\ref{fig:enter-label}, and an illustration of the transition rules in Figure~\ref{fig:all-configs}.

Recently, the Cluster-Cluster model in one dimension was investigated in \cite{berger2025onedimensionalclusterclustermodel}.
In this paper we provide a mathematical treatment of the Cluster-Cluster model in higher dimensions. We investigate fundamental properties, such as the formation of infinite clusters and the size of a typical cluster.

\subsection{Definition of the Cluster-Cluster model}
In order to define the Cluster-Cluster model (also called Cluster-Cluster Aggregation), we require some setup.
Denote by $\mathcal{E}(\BZ^d)$ the set of nearest neighbor edges of $\BZ^d$. 
We have the following definition of a cluster. 

\begin{definition} Let $G=(\mathfs{V},\mathfs{E})$ be a subgraph of $(\Z^d,\mathcal{E}(\Z^d))$ which is a forest. 
A graph $\mathfs{C}=(\mathfs{C}_V,\mathfs{C}_E) \subseteq \mathfs{V} \times \mathfs{E}$ is called a cluster in $G$, if it is a tree, and for any $w \in \mathfs{V} \setminus \mathfs{C}_V$ and $v \in \mathfs{C}_V$ with $\{ v,w \} \in \mathcal{E}(\Z^d)$, we have that $\{ v,w \} \notin \mathfs{E}$. In words, $\mathfs{C}$ is a maximal tree in the graph $G$.  We denote by $\partial \mathfs{C}$ the $\BZ^d$-edge boundary of the vertices in $\mathfs{C}$.
\end{definition}

\begin{figure}
    \centering
    \includegraphics[width=0.5\linewidth]{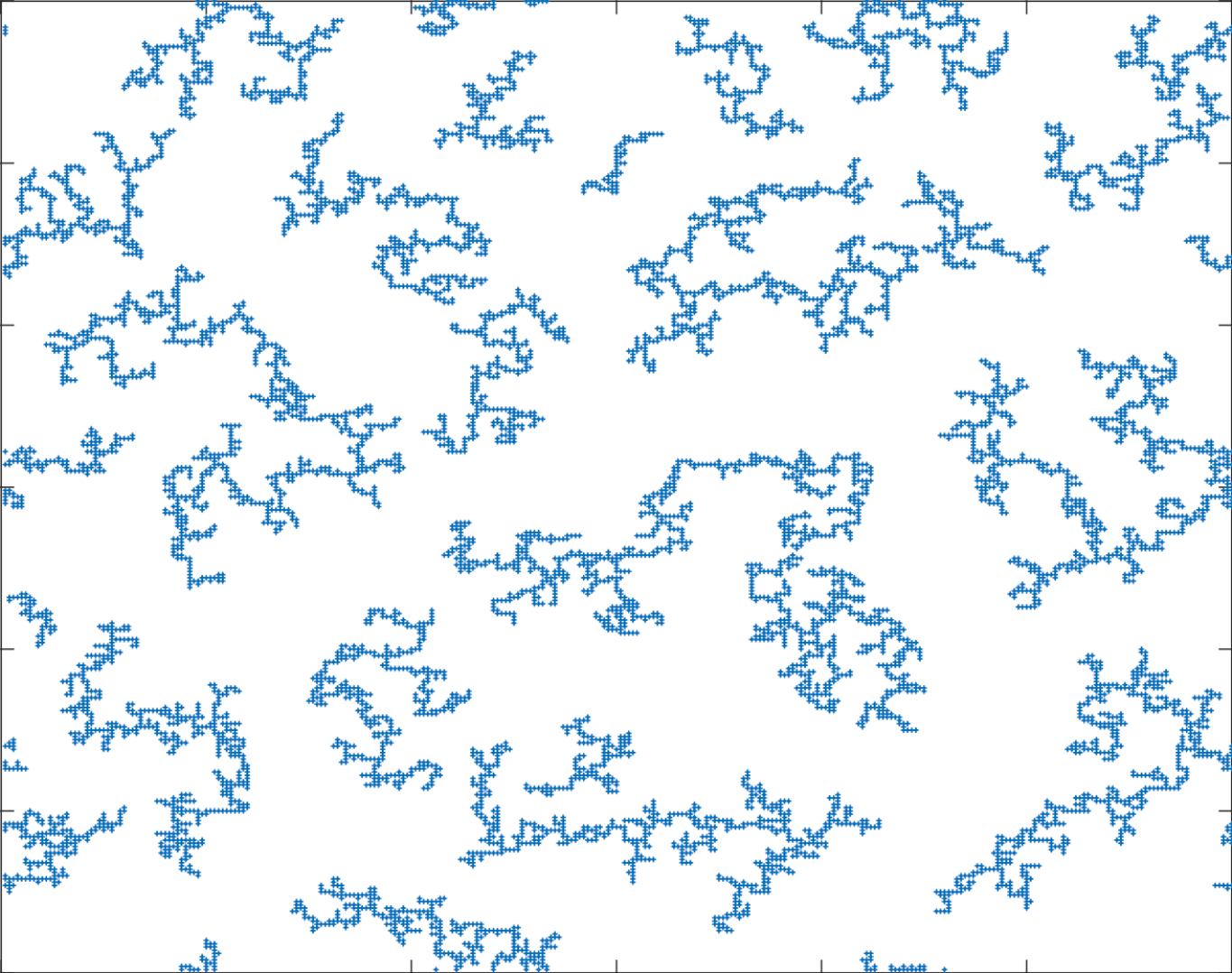}
    \caption{Cluster-Cluster model on a torus with $\alpha=1$.}
    \label{fig:enter-label}
\end{figure}

We define the Cluster-Cluster model as a stochastic process ${\rb{\mathfs{A}_t}}_{t\ge 0}$, taking values $\{0,1\}^{\BZ^d}\times \{0,1\}^{\mathcal{E}(\BZ^d)}$, where we set $${\rb{\mathfs{A}_t}}_{t\ge 0}=\rb{\mathfs{V}_t, \mathfs{E}_t}_{t\ge 0}. $$ 
We have the following formal construction. 

\begin{definition}\label{def:clustercluster} We define the Cluster-Cluster model
$\rb{\mathfs{A}_t^\alpha}_{t\ge 0}$ with parameter $\alpha \in \BR$  as follows. 
Suppose at time $0$, we are given a forest $\mathfs{A}_0^\alpha=(\mathfs{V}_0,\mathfs{E}_0)$, containing only finite trees a.s. 
For $t\geq 0$, every cluster $\mathfs{C}$ in the graph $\rb{\mathfs{V}_t,\mathfs{E}_t}$ receives a rate $|\mathfs{C}|^{-\alpha}$ Poisson clock.
When the clock of a cluster $\mathcal{C}$ rings at time $t$, the cluster attempts to perform a random walk move, i.e.~we consider the cluster $\mathcal{C}'=(\mathcal{C}_{V}',\mathcal{C}_{E}')$ with vertex set
\begin{equation*}
    \mathcal{C}_{V}' = \mathcal{C}_{V} + r := \{ v'\in  \BZ^d \colon  v'=v+r, v\in\mathcal{C}_{V}   \}
\end{equation*}
 and edge set 
\begin{equation*}
    \mathcal{C}_{E}' = \mathcal{C}_{E} + r := \{ e'\in  \mathcal{E}(\BZ^d) \colon  e'=e+r , e\in\mathcal{C}_{E} \} , 
\end{equation*} where $r$ is chosen uniformly at random from the unit vectors $(\pm\textup{e}_i)_{i \in 1,\dots,d}$ in $\Z^d$.
 If the target is vacant, i.e. the cluster $(\mathcal{C}_{V}',\mathcal{C}_{E}')$ satisfies
 \begin{equation}\label{eq:Intersection}
   \mathcal{C}_{V}' \cap \tilde{\mathcal{C}}_V = \emptyset
 \end{equation} for all clusters $\tilde{\mathcal{C}}=(\tilde{\mathcal{C}}_V,\tilde{\mathcal{C}}_E)\neq \mathcal{C}$ in $\mathfs{A}_{t_-}^\alpha$, then replace $\mathcal{C}$ by $\mathcal{C}'$ in $\mathfs{A}_{t}^\alpha$. Otherwise, let $\mathfs{I}_{V}$ denote the union of all  $\tilde{\mathcal{C}}_V$ so that the intersection in \eqref{eq:Intersection} is non-empty, and let
 \begin{equation*}
    \partial C^r := \{ e \in  \mathcal{E}(\Z^d)  \colon e=\{v,w\} , v\in \mathcal{C}_{V} , w \in \mathfs{I}_{V} \}
 \end{equation*}
denote the set of all edges to vertices, blocking the cluster $\mathcal{C}$ from moving in direction $r$. Pick an edge $\tilde{e} \in \partial C^r$ uniformly at random, and add it to $\mathcal{C}$ (this results in merging the cluster $\mathcal{C}$ with another cluster connected to $\mathcal{C}$ via $r$). All remaining clusters are left unchanged.
\end{definition}
\begin{remark}\label{rem:genconinfclsrforflr}
    For generator continuity purposes, for $\alpha \geq 0$ we expand the definition for infinite clusters. If $\alpha=0$ and infinite cluster has rate 1, and if $\alpha>0$ it has rate 0. In case of a ring in direction $r$, if $| \partial C^r|<
    \infty$ choose an edge uniformly at random, and if $| \partial C^r|= \infty$, then make no connections.
    For $\alpha \geq 0$, this will allow us to prove that the process is Feller (see Proposition \ref{prop:flrlgtebpnbds} below). For $\alpha < 0$ the process is anyways not Feller, so we do not attempt to make the generator continuous.
\end{remark}

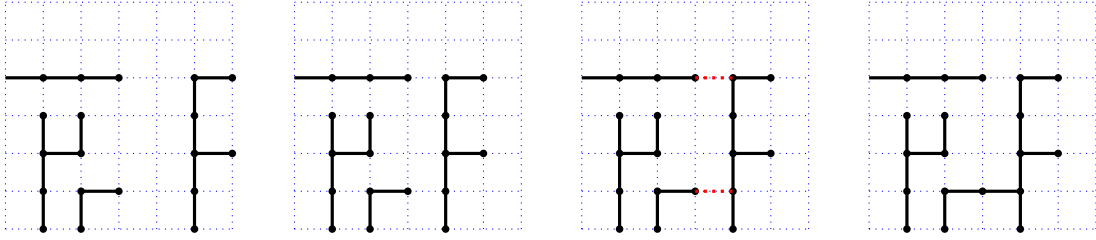
\begin{figure}[h]
    \centering

    \begin{minipage}{0.24\textwidth}
        \centering
        \begin{tikzpicture}[yscale=0.5,xscale=0.5]
            \draw[step=1cm,blue,thin, dotted] (0,0) grid (6,6);
            \node at (1,0) [circle,fill=black,inner sep=1pt]{};
            \draw [black,very thick] (1,0) -- (1,1);
            \node at (1,1) [circle,fill=black,inner sep=1pt]{};
            \draw [black,very thick] (1,1) -- (1,2);
            \node at (1,2) [circle,fill=black,inner sep=1pt]{};
            \draw [black,very thick] (1,2) -- (1,3);
            \node at (1,3) [circle,fill=black,inner sep=1pt]{};
            \draw [black,very thick] (1,2) -- (2,2);
            \node at (2,2) [circle,fill=black,inner sep=1pt]{};
            \draw [black,very thick] (2,2) -- (2,3);
            \node at (2,3) [circle,fill=black,inner sep=1pt]{};
            \node at (2,0) [circle,fill=black,inner sep=1pt]{};
            \draw [black,very thick] (2,0) -- (2,1);
            \node at (2,1) [circle,fill=black,inner sep=1pt]{};
            \draw [black,very thick] (2,1) -- (3,1);
            \node at (3,1) [circle,fill=black,inner sep=1pt]{};
            \node at (6,2) [circle,fill=black,inner sep=1pt]{};
            \node at (5,0) [circle,fill=black,inner sep=1pt]{};
            \node at (5,1) [circle,fill=black,inner sep=1pt]{};
            \draw [black,very thick] (5,0) -- (5,1);
            \draw [black,very thick] (5,1) -- (5,2);
            \node at (5,2) [circle,fill=black,inner sep=1pt]{};
            \draw [black,very thick] (5,2) -- (6,2);
            \draw [black,very thick] (5,2) -- (5,3);
            \node at (5,3) [circle,fill=black,inner sep=1pt]{};
            \draw [black,very thick] (5,3) -- (5,4);
            \node at (5,4) [circle,fill=black,inner sep=1pt]{};
            \draw [black,very thick] (5,4) -- (6,4);
            \node at (6,4) [circle,fill=black,inner sep=1pt]{};
            \draw [black,very thick] (0,4) -- (1,4);
            \node at (1,4) [circle,fill=black,inner sep=1pt]{};
            \draw [black,very thick] (1,4) -- (2,4);
            \node at (2,4) [circle,fill=black,inner sep=1pt]{};
            \draw [black,very thick] (2,4) -- (3,4);
            \node at (3,4) [circle,fill=black,inner sep=1pt]{};
        \end{tikzpicture}
    \end{minipage}%
    \hfill
    \begin{minipage}{0.24\textwidth}
        \centering
        \begin{tikzpicture}[yscale=0.5,xscale=0.5]
            \draw[step=1cm,blue,thin, dotted] (0,0) grid (6,6);
            \node at (1,0) [circle,fill=black,inner sep=1pt]{};
            \draw [black,very thick] (1,0) -- (1,1);
            \node at (1,1) [circle,fill=black,inner sep=1pt]{};
            \draw [black,very thick] (1,1) -- (1,2);
            \node at (1,2) [circle,fill=black,inner sep=1pt]{};
            \draw [black,very thick] (1,2) -- (1,3);
            \node at (1,3) [circle,fill=black,inner sep=1pt]{};
            \draw [black,very thick] (1,2) -- (2,2);
            \node at (2,2) [circle,fill=black,inner sep=1pt]{};
            \draw [black,very thick] (2,2) -- (2,3);
            \node at (2,3) [circle,fill=black,inner sep=1pt]{};
            \node at (2,0) [circle,fill=black,inner sep=1pt]{};
            \draw [black,very thick] (2,0) -- (2,1);
            \node at (2,1) [circle,fill=black,inner sep=1pt]{};
            \draw [black,very thick] (2,1) -- (3,1);
            \node at (3,1) [circle,fill=black,inner sep=1pt]{};
            \node at (4,0) [circle,fill=black,inner sep=1pt]{};
            \node at (4,1) [circle,fill=black,inner sep=1pt]{};
            \draw [black,very thick] (4,0) -- (4,1);
            \draw [black,very thick] (4,1) -- (4,2);
            \node at (4,2) [circle,fill=black,inner sep=1pt]{};
            \draw [black,very thick] (4,2) -- (5,2);
            \node at (5,2) [circle,fill=black,inner sep=1pt]{};
            \draw [black,very thick] (4,2) -- (4,3);
            \node at (4,3) [circle,fill=black,inner sep=1pt]{};
            \draw [black,very thick] (4,3) -- (4,4);
            \node at (4,4) [circle,fill=black,inner sep=1pt]{};
            \draw [black,very thick] (4,4) -- (5,4);
            \node at (5,4) [circle,fill=black,inner sep=1pt]{};
            \draw [black,very thick] (0,4) -- (1,4);
            \node at (1,4) [circle,fill=black,inner sep=1pt]{};
            \draw [black,very thick] (1,4) -- (2,4);
            \node at (2,4) [circle,fill=black,inner sep=1pt]{};
            \draw [black,very thick] (2,4) -- (3,4);
            \node at (3,4) [circle,fill=black,inner sep=1pt]{};
        \end{tikzpicture}
    \end{minipage}%
    \hfill
    \begin{minipage}{0.24\textwidth}
        \centering
        \begin{tikzpicture}[yscale=0.5,xscale=0.5]
            \draw[step=1cm,blue,thin, dotted] (0,0) grid (6,6);
            \node at (1,0) [circle,fill=black,inner sep=1pt]{};
            \draw [black,very thick] (1,0) -- (1,1);
            \node at (1,1) [circle,fill=black,inner sep=1pt]{};
            \draw [black,very thick] (1,1) -- (1,2);
            \node at (1,2) [circle,fill=black,inner sep=1pt]{};
            \draw [black,very thick] (1,2) -- (1,3);
            \node at (1,3) [circle,fill=black,inner sep=1pt]{};
            \draw [black,very thick] (1,2) -- (2,2);
            \node at (2,2) [circle,fill=black,inner sep=1pt]{};
            \draw [black,very thick] (2,2) -- (2,3);
            \node at (2,3) [circle,fill=black,inner sep=1pt]{};
            \node at (2,0) [circle,fill=black,inner sep=1pt]{};
            \draw [black,very thick] (2,0) -- (2,1);
            \node at (2,1) [circle,fill=black,inner sep=1pt]{};
            \draw [black,very thick] (2,1) -- (3,1);
            \node at (3,1) [circle,fill=black,inner sep=1pt]{};
            \node at (4,0) [circle,fill=black,inner sep=1pt]{};
            \node at (4,1) [circle,fill=black,inner sep=1pt]{};
            \draw [black,very thick] (4,0) -- (4,1);
            \draw [black,very thick] (4,1) -- (4,2);
            \node at (4,2) [circle,fill=black,inner sep=1pt]{};
            \draw [black,very thick] (4,2) -- (5,2);
            \node at (5,2) [circle,fill=black,inner sep=1pt]{};
            \draw [black,very thick] (4,2) -- (4,3);
            \node at (4,3) [circle,fill=black,inner sep=1pt]{};
            \draw [black,very thick] (4,3) -- (4,4);
            \node at (4,4) [circle,fill=black,inner sep=1pt]{};
            \draw [black,very thick] (4,4) -- (5,4);
            \node at (5,4) [circle,fill=black,inner sep=1pt]{};
            \draw [black,very thick] (0,4) -- (1,4);
            \node at (1,4) [circle,fill=black,inner sep=1pt]{};
            \draw [black,very thick] (1,4) -- (2,4);
            \node at (2,4) [circle,fill=black,inner sep=1pt]{};
            \draw [black,very thick] (2,4) -- (3,4);
            \node at (3,4) [circle,fill=black,inner sep=1pt]{};
            \draw [red,dotted,very thick] (3,1) -- (4,1);
            \draw [red,dotted,very thick] (3,4) -- (4,4);
        \end{tikzpicture}
    \end{minipage}%
    \hfill
    \begin{minipage}{0.24\textwidth}
        \centering
        \begin{tikzpicture}[yscale=0.5,xscale=0.5]
            \draw[step=1cm,blue,thin, dotted] (0,0) grid (6,6);
            \node at (1,0) [circle,fill=black,inner sep=1pt]{};
            \draw [black,very thick] (1,0) -- (1,1);
            \node at (1,1) [circle,fill=black,inner sep=1pt]{};
            \draw [black,very thick] (1,1) -- (1,2);
            \node at (1,2) [circle,fill=black,inner sep=1pt]{};
            \draw [black,very thick] (1,2) -- (1,3);
            \node at (1,3) [circle,fill=black,inner sep=1pt]{};
            \draw [black,very thick] (1,2) -- (2,2);
            \node at (2,2) [circle,fill=black,inner sep=1pt]{};
            \draw [black,very thick] (2,2) -- (2,3);
            \node at (2,3) [circle,fill=black,inner sep=1pt]{};
            \node at (2,0) [circle,fill=black,inner sep=1pt]{};
            \draw [black,very thick] (2,0) -- (2,1);
            \node at (2,1) [circle,fill=black,inner sep=1pt]{};
            \draw [black,very thick] (2,1) -- (3,1);
            \node at (3,1) [circle,fill=black,inner sep=1pt]{};
            \node at (4,0) [circle,fill=black,inner sep=1pt]{};
            \node at (4,1) [circle,fill=black,inner sep=1pt]{};
            \draw [black,very thick] (4,0) -- (4,1);
            \draw [black,very thick] (4,1) -- (4,2);
            \node at (4,2) [circle,fill=black,inner sep=1pt]{};
            \draw [black,very thick] (4,2) -- (5,2);
            \node at (5,2) [circle,fill=black,inner sep=1pt]{};
            \draw [black,very thick] (4,2) -- (4,3);
            \node at (4,3) [circle,fill=black,inner sep=1pt]{};
            \draw [black,very thick] (4,3) -- (4,4);
            \node at (4,4) [circle,fill=black,inner sep=1pt]{};
            \draw [black,very thick] (4,4) -- (5,4);
            \node at (5,4) [circle,fill=black,inner sep=1pt]{};
            \draw [black,very thick] (0,4) -- (1,4);
            \node at (1,4) [circle,fill=black,inner sep=1pt]{};
            \draw [black,very thick] (1,4) -- (2,4);
            \node at (2,4) [circle,fill=black,inner sep=1pt]{};
            \draw [black,very thick] (2,4) -- (3,4);
            \node at (3,4) [circle,fill=black,inner sep=1pt]{};
            \draw [black,very thick] (3,1) -- (4,1);
        \end{tikzpicture}
    \end{minipage}%
    \caption{Example of move step and connection step evolving from left to right.}
    \label{fig:all-configs}
\end{figure}

Note that in dimension $1$, the boundary of clusters are always of size $2$, thus when a cluster moves, it can encounter at most one other cluster. Thus, one does not need a mechanism to choose connections such as in $d>1$.

We will impose throughout this article that $\mathfs{A}_0$ is sampled from a stationary and ergodic measure with marginal distribution $\prob(\mathfs{A}_0(v)=1)=p$, i.e., a distribution over $\{0,1\}^{\BZ^d}\times \{0,1\}^{\mathcal{E}(\BZ^d)}$ which is invariant under all $\BZ^d$ translations. An important special case is i.i.d.~site percolation with parameter $p\in(0,1]$ for the vertices, and an empty edge set at initiation i.e. $\mathfs{E}_0=\emptyset$.

Let us stress that the rate at which clusters move changes over time. In particular, as we will see in the sequel, this may create infinite clusters in finite time, for which Definition~\ref{def:clustercluster} no longer make sense. This is addressed in the following definition of well-definedness for the Cluster-Cluster model, and blowups.

\begin{definition}
Whenever for some parameters $\alpha$ and a distribution for an initial configuration there is no c\`adl\`ag process on $\{0,1\}^{\BZ^d}\times \{0,1\}^{\mathcal{E}(\BZ^d)}$ satisfying Definition \ref{def:clustercluster}, we say that the Cluster-Cluster model is not well defined for these parameters. If there exists some $t_b>0$ such that the process is well defined for all $t<t_b$ and not well defined for $t=t_b$, we say that the process has a blowup at time $t_b$. We call a blowup immediate if the Cluster-Cluster model is not well-defined for any time $t>0$. We say that the Cluster-Cluster model exhibits an eventual blowup if we see a blowup in finite time, but no immediate blowup. 
\end{definition}


\subsection{Main results}

Our main results concern the long-term behavior of the Cluster-Cluster model on $\BZ^d$. To this end, denote by $\mathfs{C}_0(t)$ the cluster started closest to $0$ at time $t$.
The first result claims that there are no infinite clusters in the no-speedup regime. 
\begin{theorem}\label{thm:finite_clusters}
For every \(d\ge 1\), every \(\alpha\ge 0\), and every stationary and ergodic starting configuration $\mathfs{A}_0$,
\[
    |\mathfs{C}_t(0)|<\infty
    \qquad\text{for all }t<\infty \text{ almost surely.}
\]
\end{theorem}
In Appendix \ref{appendix}, we present an alternative proof of Theorem~\ref{thm:finite_clusters} which is weaker in the sense that it shows non-explosion only for $\alpha\ge1$. However, we include it because we believe it is elegant and teaches one about the cluster formation structure.

The second result assures that for sufficiently negative $\alpha$ there is a blowup. 
\begin{theorem}
For any stationary and ergodic starting configuration the following hold.
\begin{enumerate}
  \item  Let $d> 1, \alpha<-1, p=1$. Then there is almost surely an immediate blowup.
  \item  Let $d> 1, \alpha=-1, p=1$. Then there is almost surely a blowup in finite time.
  \item  Let $d>1, \alpha<-1-\frac{2}{d}, p\in(0,1)$. Then there is almost surely a blowup in finite time.
\end{enumerate}
\label{thm:blowupregieme}\end{theorem}

Next we show that in the regime $\alpha\in(-1,0)$ one can construct stationary and ergodic starting configurations that lead to immediate explosion and configurations that never explode a.s.
\begin{theorem}\label{bad_start_config}
    For every $p\in(0,1]$, $\alpha<0$, there exists a translation invariant and ergodic starting condition, such that there exists an infinite cluster almost surely at any strictly positive time.
\end{theorem}

\begin{theorem}\label{thm:good_start_config}
    For every $p\in(0,1]$, $\alpha>-1$
    there exists a translation invariant and ergodic starting configuration such that for every $t>0$ all clusters are finite a.s.
\end{theorem}
\begin{definition}
  Let $\mathcal{S}^{(p)}$ denote the class of all stationary and ergodic initial configurations for the Cluster-Cluster model with density parameter $p$. We define
  \begin{equation}
      \begin{split}
          \alpha_{c,p}^+&:=\sup_{\mathfs{A}_0 \in \mathcal{S}^{(p)}}\sup\{\alpha \in \BR:\mathfs{A}^{\alpha}_t\text{ blows up for some finite }t \text{ almost surely}\}  \\
          \alpha_{c,p}^-&:=\inf_{\mathfs{A}_0 \in \mathcal{S}^{(p)}}\inf\{\alpha \in \BR:\mathfs{A}^{\alpha}_t\text{ has no blowup}\} . 
      \end{split}
  \end{equation} as the upper and lower critical value for blowup of the Cluster-Cluster model. 
\end{definition}
As an immediate consequence of Theorems \ref{thm:finite_clusters} and \ref{bad_start_config}, we get that the upper and lower critical value of the Cluster-Cluster model in $d$ dimensions coincide. 
\begin{corollary} For any $p\in (0,1]$ and $d>1$, we have that 
$\alpha^+_{c,p}=\alpha^-_{c,p}=0$.
\end{corollary}


In the fully packed case one-dimensional case, due to an astounding independence between cluster sizes, we obtain rigorous exact Smoluchowski type equations that allow us to present exact results about the phase diagram and growth bounds. 
Let the gelation time be defined by 
\begin{equation}\label{def:Gelation} 
    T_\alpha:=\sup\{t\ge0: \mathbb{E}[|\mathfs{C}_0(t)|]<\infty\}.
\end{equation}

\begin{theorem}
\label{thm:one-dimensional-phase-diagram}
For the fully occupied one-dimensional Cluster-Cluster model ($p=1,d=1$), the following bounds hold:
\begin{enumerate}
\item If \(\alpha>0\), then \(T_\alpha=\infty\) and
\[
   \mu(t) := \mathbb{E}[|\mathfs{C}_0(t)|]\asymp (1+t)^{1/\alpha}.
 \]
\item If \(\alpha=0\), then \(T_0=\infty\) and 
\(\mu(t)=e^t\). 
\item If \(-1<\alpha<0\), and let $q=-\alpha$, then \(0<T_\alpha<\infty\). More explicitly,
 \(\exists \kappa_q>0\) such that
\begin{equation*}
\label{eq:gel-time-bounds-sublinear}
    \frac1q
    \le
    T_q
    \le
    \min\left\{
        \frac1{\kappa_q},
        \frac1{1-2^{-q}}
    \right\}.
\end{equation*}

\item If \(\alpha=-1\), then \(T_{-1}=1\) and
\(\mu(t)=(1-t)^{-1}\).
\item If \(\alpha<-1\), then \(T_\alpha=0\).
\end{enumerate}
\end{theorem}




\subsection{Related work}

Note that even for simpler aggregation processes than the Cluster-Cluster model with a single aggregate, not much is currently known about growth rates. For DLA, the best result is still Kesten's \cite{kesten1987long}, where it is proved that DLA with $n$ particles is contained in a ball of radius $n^{2/3}$. Similar results were recently established for the dielectric breakdown model \cite{losev2025long}. The only DLA type model in which we have exact growth bounds is an off-lattice version constructed by a composition of conformal maps called the stationary Hastings-Levitov(0) \cite{berger2022growth,berger2025logarithmic,procaccia2021dimension}.  For Multi-particle DLA (MDLA), a model which has a single aggregate but driven by a field of random walkers, much as our model, in high dimensions a linear growth rate is known in the high intensity regime. For 1d MDLA, exact growth bounds are known, with interesting critical behaviour \cite{elboim2020critical,sidoravicius2017one,sly2020one}. Mean field versions of our model were considered by employing equations resembling the Smoluchowski equations \cite{smoluchowski1918versuch}, as part of the study of Ostwald ripening \cite{ostwald1903lehrbuch,voorhees1985theory}. For the special case $p=1$, exact cluster size distribution is known \cite{pitman1999coalescent,aldous1998standard}, as the Cluster-Cluster model on the complete graph (random walk steps are taken by choosing a uniform random permutation) is exactly the standard additive coalescent with kernel $K(i,j)=i+j$. Mainly, the number of clusters at time $t>0$ is distributed $\text{Bin}(n,e^{-t})$. Conditional on having $k$ clusters, the clusters are distributed according to the uniform spanning forest with $k$ trees. For general $p\in(0,1]$ and $\alpha\in\BR$, the Cluster-Cluster dynamics on the complete graph corresponds to the Marcus-Lushnikov coalescent with kernel $K_\alpha(i,j)=i^{-\alpha}j+j^{-\alpha}i$. The expected growth rate in this case is
\[
    |\mathfs{C}_0(t)|\asymp
    \begin{cases}
        t^{1/\alpha}, & \alpha>0,\\[4pt]
        e^{ct}, & \alpha=0,\\[4pt]
        (T_{\mathrm{gel}}-t)^{-1/|\alpha|}, & \alpha<0,
    \end{cases}
\]
where $T_{\mathrm{gel}}$ is the blowup time or the gelation time as it is called in the coalescent literature. In Figure \ref{fig:linear}, one can see that for $\alpha>0$ indeed the conjectured mean field growth rate holds for $\BZ^d$. 
\begin{figure}
    \centering
    \includegraphics[width=0.5\linewidth]{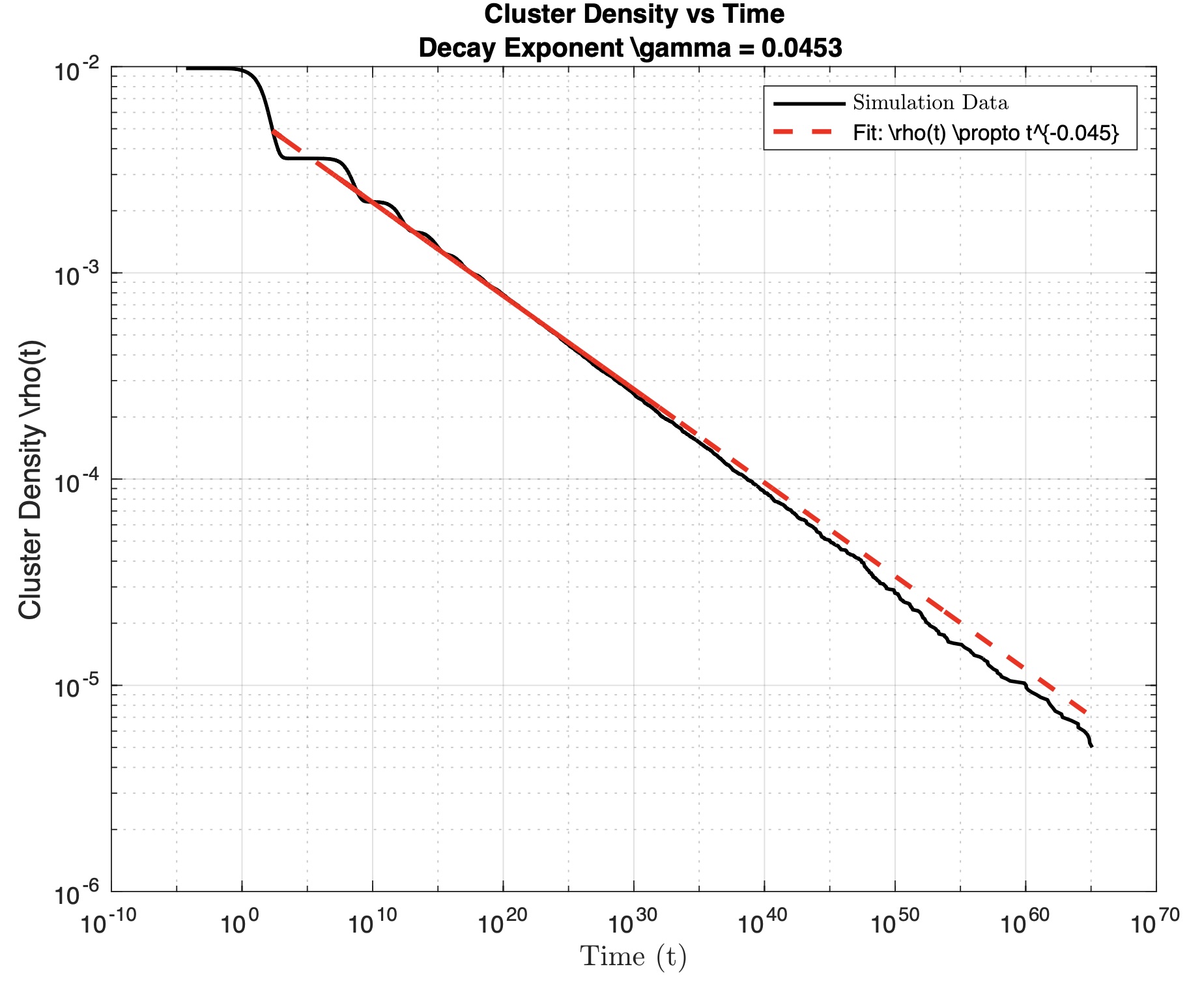}
    \caption{Simulation in $\BZ^2$  showing $t^{1/\alpha}$ growth rate for $\alpha=20$.}
    \label{fig:linear}
\end{figure}

Note also that some spatial versions of
Smoluchowski's coagulation equation have been analyzed in the literature
\cite{kryven2026spatial}.

The one-dimensional Cluster-Cluster model with \(p<1\) was recently studied in
\cite{berger2025onedimensionalclusterclustermodel}. This regime is
substantially different from the other regimes considered here. Indeed, in
dimension one, finite clusters are intervals and therefore have full dimension:
to interact, two clusters must first diffuse
across the vacant gaps separating them. Thus the relevant time scale contains
both the internal clock scale of a cluster of size \(m\), namely \(m^\alpha\),
and the diffusive time \(m^2\) needed to find another cluster at the typical
spacing. This additional diffusive cost changes the growth exponent. In
particular, the characteristic cluster size grows for all $\alpha>-2$ on the scale
\[
    t^{1/(\alpha+2)},
\]
rather than on the \(t^{1/\alpha}\)-type scale appearing in the fully occupied
one-dimensional case and in the higher-dimensional regimes discussed in this
paper. This suggests a broader heuristic: in variants of the Cluster-Cluster
model where clusters are constrained to remain ball-like, so that growth
requires the motion of full-dimensional bodies across gaps, one should again
expect the diffusive cost to enter the scaling and the growth rate to be of
order \(t^{1/(\alpha+2)}\). See the Pool aggregation model \cite{cai2026pool} for a possible construction of a full dimension and mass preserving aggregate. 

The Cluster-Cluster model on $\Z^d$ is also related to the kinetic-limit work of Hammond and Rezakhanlou
\cite{HammondRezakhanlou2007KineticLimit}, where coagulating Brownian point
particles with mass-dependent diffusivity converge to the spatial Smoluchowski
PDE. In the hard-core point-particle setting the effective kernel is of order
\(d(m)+d(n)\); thus, if \(d(m)\asymp m^{-\alpha}\), the corresponding
mean-field heuristic gives a typical mass scale \(t^{1/(\alpha+1)}\). This
differs from the sparse one-dimensional Cluster-Cluster model, where clusters are
extended intervals and must diffuse across vacant gaps, producing the scale
\(t^{1/(\alpha+2)}\). The distinction is that in the present model the geometry of the
cluster and the gaps between clusters are part of the dynamics, whereas in the
Smoluchowski limit the particles remain point masses. We also note the related
moment and mass-conservation results for the Smoluchowski PDE in
\cite{HammondRezakhanlou2007MomentBounds}, and Hammond's survey
\cite{Hammond2014SmoluchowskiSurvey}.

\subsection{Open Problems}

We now provide a selection of open problems concerning a phase transition for the occurrence of blowups, the growth rate of clusters, and the eventual local freezing of Cluster-Cluster model. 

\begin{open_problem}
Show that for all $d>1$, when initializing with i.i.d. site percolation, there exist some $-1<\alpha_\ast\le\alpha_c<0$, such that for all $\alpha>\alpha_c$ all clusters are finite a.s., for all $\alpha\in[\alpha_\ast,\alpha_c]$ there is eventual blowup, and for all $\alpha<\alpha_\ast$ there is immediate blowup. 
\end{open_problem}
\begin{open_problem}
Show that for all $d>1$, and any $p\in(0,1]$, $\alpha>0$, $|\mathfs{C}_0(t)|\asymp t^{1/\alpha}$.
\end{open_problem}
\begin{open_problem}
For every $\alpha>\alpha_c$, for every $p\in(0,1]$, for every $M>0$ there is an a.s. finite time $t_M$ such that the Cluster-Cluster configuration inside $[-M,M]^d$ is fixed for all $t>t_M$. That is $\mathfs{A}_{t_1}^\alpha\cap [-M,M]^d=\mathfs{A}_{t_2}^\alpha\cap [-M,M]^d$, $\forall t_1, t_2>t_M$ a.s.
\end{open_problem}


\subsection{Organization of the paper}
This paper is organized as follows. In Section \ref{sec:well} we construct the
cluster-cluster dynamics and discuss well-definedness of the process in the
different regimes of the parameter \(\alpha\). Section \ref{sec:no_blow} proves the
non-explosion result for the no-speedup regime \(\alpha\ge0\), using a
density-flux argument. In Section \ref{sec:blowup} we turn to the blowup regime and prove
finite time, and in some cases immediate, blowup for sufficiently negative
\(\alpha\), beginning with the fully occupied case \(p=1\). Section \ref{sec:vlowp<1} treats
the sparse case \(p<1\) and proves blowup for \(\alpha<-1-2/d\). In
Section \ref{sec:startconfig} we construct special stationary initial configurations, showing that
in the intermediate negative \(\alpha\) regime the occurrence of blowup can
depend on the initial geometry. Finally, Section \ref{sec:fullypacked} studies the fully packed
one-dimensional model. We prove the sharp growth bounds for
\(\alpha>0\) and analyze the negative-\(\alpha\) regime in dimension one. 
Appendix~\ref{appendix} contains an alternative proof of non-explosion in part of the
no-speedup regime, revealing structural insight into the cluster
formation mechanism.

\section{Well-definedness of the Cluster-Cluster model}\label{sec:well}
Before discussing the explosion of the Cluster-Cluster model, 
we first argue about the well-definedness for different regimes of $\alpha \in \BR$. We start with the case where $\alpha \geq 0$, where the rates are bounded uniformly from above, and where we can follow a standard approach from \cite{Liggett1985InteractingParticleSystems} to show that the process is well defined. In all other regimes, we argue that determining whether the Cluster-Cluster model is well-defined requires a more delicate case by case analysis, which we defer to later points in this manuscript.

We recall the standard notation used in Liggett's construction of interacting
particle systems. Let \(S\) be a finite set and let
\[
    \Omega=S^{\mathbb Z^d}
\]
be equipped with the product topology. For \(f\in C(\Omega)\) and
\(x\in\mathbb Z^d\), set
\[
    \Delta_f(x):=
    \sup\{|f(\eta)-f(\xi)|:\eta(y)=\xi(y)\text{ for all }y\neq x\}.
\]
Following \cite[Chapter I, Section 3]{Liggett1985InteractingParticleSystems},
define
\[\label{eq:LipschitzTriple}
     \vvvert f  \vvvert:=\sum_{x\in\mathbb Z^d}\Delta_f(x),
    \qquad
    D(\Omega):=\{f\in C(\Omega):\vvvert f  \vvvert<\infty\}.
\]
The class \(D(\Omega)\) contains all local functions and is dense in
\(C(\Omega)\).

We encode the cluster-cluster model as an interacting particle system with a
finite single-site state space. Let
\[
    S=\{0,1\}^{\{0,\pm \textup{e}_1,\ldots,\pm \textup{e}_d\}}.
\]
For \(\eta\in\Omega=S^{\mathbb Z^d}\), the coordinate \(\eta_x(0)\) records
whether \(x\) is occupied, while \(\eta_x(e)\) records whether the oriented
edge from \(x\) to \(x+e\) is open. We restrict to the closed set of admissible
configurations for which
\[
    \eta_x(e)=\eta_{x+e}(-e)
\]
and an open edge can occur only between occupied vertices. The clusters are the
connected components of the graph formed by occupied vertices and open edges.

When a finite cluster \(C\) rings, it chooses one of the \(2d\) nearest-neighbour
directions uniformly at random, attempts to move one step in that direction, and then
coalesces with any cluster it encounters, according to the definition of the
model. For $\alpha \geq 0$ infinite clusters, if present, are treated as in Remark \ref{rem:genconinfclsrforflr}. For $\alpha < 0$ when an infinite cluster emerges the process seizes to be well-defined. This convention is
irrelevant before the possible formation of an infinite cluster, and it gives a
globally defined dynamics on the compact state space. 
We use the following theorem by Liggett:

\begin{theorem}[Liggett's construction theorem; \cite{Liggett1985InteractingParticleSystems}, Chapter I, Proposition 3.2]
\label{thm:liggett-construction}
Let \(S\) be a finite set and let \(\Omega=S^{\mathbb Z^d}\), equipped with the
product topology. 
Suppose that \(L\) is a pure-jump operator on \(D(\Omega)\) of the form
\[
    Lf(\eta)=\int_\Omega \bigl(f(\xi)-f(\eta)\bigr)c(\eta,d\xi),
\]
where \(c(\eta,d\xi)\) is a transition kernel on \(\Omega\). Assume that, for
every \(f\in D(\Omega)\), \(Lf\in C(\Omega)\), and that there exists a constant
\(C_0<\infty\) such that
\[
    \sup_{\eta\in\Omega}
    \int_\Omega |f(\xi)-f(\eta)|\,c(\eta,d\xi)
    \le C_0 |||f|||
    \qquad \text{for all } f\in D(\Omega).
\]
Then \(L\) is closable on \(C(\Omega)\), and its closure generates a
conservative Feller semigroup on \(C(\Omega)\). In particular, there exists a
càdlàg Markov process on \(\Omega\) with generator extending \(L\).
\end{theorem}
\begin{proposition}\label{prop:flrlgtebpnbds}
For every \(\alpha\ge0\) and every \(d\ge1\), the Cluster-Cluster dynamics
defines a Feller process on the admissible state space. In particular, the
process is well defined for all \(t\ge0\).
\end{proposition}

\begin{proof}
For an admissible configuration \(\eta\), let \(\mathcal C(\eta)\) be the set
of finite clusters of \(\eta\). If \(C\in\mathcal C(\eta)\) and
\(e\in\{\pm \textup{e}_1,\ldots,\pm \textup{e}_d\}\), recalling Definition~\ref{def:clustercluster}, let $\partial C^e$ be the set of boundary edges of $C$ blocking movement in direction $e$. If no such edge exists, let $\partial C^e=e$. Let \(\theta_{C,e,g}\eta\) be the
configuration obtained after the cluster \(C\) attempts to jump in direction
\(e\) and if blocked connects via edge $g$. The generator is
\[
    Lf(\eta)
    =
    \sum_{C\in\mathcal C(\eta)}
    \frac{|C|^{-\alpha}}{2d}
    \sum_{e=\pm \textup{e}_1,\ldots,\pm \textup{e}_d}\sum_{g\in\partial ^eC}\frac{1}{|\partial^eC|}
    \bigl(f(\theta_{C,e,g}\eta)-f(\eta)\bigr).
\]
We show that this operator is bounded on \(D(\Omega)\) in Liggett's triple norm.

There exists a constant \(r=r(d)\) such that, for every cluster \(C\) and every
direction \(e\), and $g\in\partial^e C$, the configurations \(\eta\) and \(\theta_{C,e,g}\eta\) may
differ only on
\[
    B_r(C):=\{x\in\mathbb Z^d:\operatorname{dist}(x,C)\le r\}.
\]
For the above edge encoding one may take \(r=2\). Hence, by changing the
coordinates one at a time,
\[
    |f(\theta_{C,e,g}\eta)-f(\eta)|
    \le
    \sum_{x\in B_r(C)}\Delta_f(x).
\]
Since \(\alpha\ge0\), we have \(|C|^{-\alpha}\le1\). Therefore
\[
\begin{aligned}
    \int |f(\xi)-f(\eta)|\,c(\eta,d\xi)
    &=
    \sum_{C\in\mathcal C(\eta)}
    \frac{|C|^{-\alpha}}{2d}
    \sum_{e=\pm \textup{e}_1,\ldots,\pm \textup{e}_d}
    \sum_{g\in\partial ^eC}\frac{1}{|\partial^eC|}
    \bigl(f(\theta_{C,e,g}\eta)-f(\eta)\bigr)  \\
    &\le
    \sum_{C\in\mathcal C(\eta)}
    \sum_{x\in B_r(C)}\Delta_f(x).
\end{aligned}
\]
Now fix \(x\in\mathbb Z^d\). The number of clusters \(C\) such that
\(x\in B_r(C)\) is at most \(|B_r(0)|\), because each such cluster must contain
at least one vertex in \(B_r(x)\), and distinct clusters contain disjoint
vertices. Hence
\[
    \sum_{C\in\mathcal C(\eta)}
    \sum_{x\in B_r(C)}\Delta_f(x)
    \le
    |B_r(0)|\sum_{x\in\mathbb Z^d}\Delta_f(x)
    =
    |B_r(0)|\,\vvvert f  \vvvert.
\]
Thus, uniformly in \(\eta\),
\[
    \int |f(\xi)-f(\eta)|\,c(\eta,d\xi)
    \le
    |B_r(0)|\,\vvvert f  \vvvert.
\]
In particular the series defining \(Lf(\eta)\) is absolutely convergent for
every \(f\in D(\Omega)\), and
\[
    \|Lf\|_\infty\le |B_r(0)|\,\vvvert f  \vvvert.
\]

The operator \(L\) is of pure-jump form, satisfies \(L1=0\), and satisfies the
positive maximum principle. Therefore, by Theorem \ref{thm:liggett-construction} its closure generates a Feller semigroup
on \(C(\Omega)\). Since the admissible configurations form a closed invariant
subset, the semigroup restricts to the admissible state space. Hence for $\alpha \geq 0$ the
Cluster-Cluster model with any initial configuration is well defined for all times.
\end{proof}


For $\alpha \geq 0$, all rates are uniformly bounded by $1$, and thus fit well the general theory by Liggett. For $\alpha<0$, it remains to determine the well-definedness for different intervals of times. In some cases (for instance Theorem~\ref{thm:good_start_config}) we show that for any $\alpha \in (-1,0)$ and $p\in(0,1]$ we can find starting configurations, where one can uniformly bound the rates, and apply the arguments from above to obtain well-definedness. We conclude this discussion on well definedness of the Cluster-Cluster model with the following result.

\begin{proposition}
\label{prop:construction-until-explosion}
Let \(\alpha<0\), and write \(\beta=-\alpha>0\). The Cluster-Cluster dynamics
admits a maximal local construction as follows. For every finite
\(\Lambda\subseteq\mathbb Z^d\) there is a stopping time
\(\tau^\Lambda\in[0,\infty]\) such that the restriction of the process to
\(\Lambda\) is well defined on \([0,\tau^\Lambda)\), has càdlàg paths there,
and has only finitely many jumps on every compact subinterval of
\([0,\tau^\Lambda)\). If \((\Lambda_j)_{j\ge1}\) is an increasing exhaustion of
\(\mathbb Z^d\) by finite boxes, the maximal lifetime is
\[
    \tau_{\rm exp}:=\inf_{j\ge1}\tau^{\Lambda_j}.
\]
For every \(T<\tau_{\rm exp}\), the process is defined on every finite set up to
time \(T\). If \(\tau_{\rm exp}<\infty\), then local stabilization of the cutoff
processes fails at \(\tau_{\rm exp}\); equivalently, arbitrarily large cutoff
levels influence the evolution in finite boxes arbitrarily close to
\(\tau_{\rm exp}\).
\end{proposition}

\begin{proof}
Encode configurations as elements of a closed subset of \(S^{\mathbb Z^d}\),
where \(S\) is finite and records occupation variables and open nearest-neighbour
edges. Note that clusters are the connected components of the occupied graph.

For \(K\ge1\), define the \(K\)-truncated dynamics by freezing clusters of size
larger than \(K\). Thus a finite cluster \(C\) with \(|C|\le K\) rings at rate
\(|C|^\beta\), chooses one of the \(2d\) directions uniformly, and performs the
usual Cluster-Cluster move, while clusters of size \(>K\) do not ring. Let \(L_K\) be
the corresponding generator.

For each fixed \(K\), this is a bounded-rate interacting particle system. A
jump of a cluster \(C\) can change the configuration only in a fixed
neighbourhood \(B_r(C)\), where \(r=r(d)\). Therefore, for
\(f\in D(\Omega)\),
\[
    |L_K f(\eta)|
    \le
    K^\beta\sum_C\sum_{x\in B_r(C)}\Delta_f(x)
    \le
    K^\beta |B_r(0)|\,\vvvert f\vvvert .
\]
Indeed, each \(x\) belongs to \(B_r(C)\) for at most \(|B_r(0)|\) distinct
clusters \(C\). Hence Liggett's bounded-rate construction gives a Feller
process \((\eta_t^K)_{t\ge0}\) (alternatively we can use the proof of Proposition \ref{prop:flrlgtebpnbds}.

We now couple the truncated processes. For every finite connected
\(A\subseteq\mathbb Z^d\) and every direction \(e\), take an independent Poisson
process of rate \(|A|^\beta/(2d)\). For the \(K\)-truncated process, only the
clocks with \(|A|\le K\) are used: when the clock \((A,e)\) rings, the update is
performed only if \(A\) is exactly a cluster of the current configuration. This
is a valid construction for each fixed \(K\), since for every finite space
region and fixed \(K\), only finitely many sets \(A\) of size at most \(K\) can
affect that region. Thus all truncated processes are realized on one common
probability space. More formally, the coupling of any finite number of such processes
is realizable as a Feller process, and the existence of a coupling of all
of them together follows from Kolmogorov’s extension theorem. 

Fix \(\Lambda\subseteq\mathbb Z^d\). For \(K\ge1\), define
\[
    \sigma_K^\Lambda
    :=
    \inf\Bigl\{
        t\ge0:
        \exists K_1,K_2\ge K
        \text{ such that }
        \eta_t^{K_1}|_\Lambda\ne \eta_t^{K_2}|_\Lambda
    \Bigr\}.
\]
The stopping times \(\sigma_K^\Lambda\) are nondecreasing in \(K\), because the
set of pairs \(K_1,K_2\ge K\) decreases as \(K\) increases. Set
\[
    \tau^\Lambda:=\lim_{K\to\infty}\sigma_K^\Lambda .
\]
If \(t<\tau^\Lambda\), choose \(K\) such that \(t<\sigma_K^\Lambda\), and define
\[
    \eta_t|_\Lambda:=\eta_t^K|_\Lambda .
\]
This is independent of the choice of \(K\), i.e., if \(K',K''\ge K\), then
\(\eta_t^{K'}|_\Lambda=\eta_t^{K''}|_\Lambda\) by the definition of
\(\sigma_K^\Lambda\).

Moreover, if \(T<\tau^\Lambda\), then for some \(K\) we have
\(T<\sigma_K^\Lambda\). Hence, on the whole interval \([0,T]\), the local
process in \(\Lambda\) agrees with the bounded-rate process \(\eta^K\). It
therefore has càdlàg paths and only finitely many jumps on \([0,T]\). Since
\(T<\tau^\Lambda\) was arbitrary, the local process is well defined on
\([0,\tau^\Lambda)\).

Finally let \((\Lambda_j)_{j\ge1}\) be an increasing exhaustion of
\(\mathbb Z^d\), and define
\[
    \tau_{\rm exp}:=\inf_{j\ge1}\tau^{\Lambda_j}.
\]
For every \(T<\tau_{\rm exp}\) and every finite \(\Lambda\), choose \(j\) with
\(\Lambda\subseteq\Lambda_j\). Then \(T<\tau^{\Lambda_j}\), so the process is
defined on \(\Lambda_j\), and hence on \(\Lambda\), up to time \(T\). Thus the
process is globally defined on \([0,\tau_{\rm exp})\) as a local limit.

If \(\tau_{\rm exp}<\infty\), then by definition there are finite boxes
\(\Lambda_j\), arbitrarily large cutoff levels \(K\), and times approaching
\(\tau_{\rm exp}\) for which the restrictions of two cutoff processes with
cutoffs at least \(K\) disagree in \(\Lambda_j\). In other words, the local
limit of the cutoff dynamics fails to stabilize at \(\tau_{\rm exp}\). This is
the explosion alternative for the maximal construction.
\end{proof}

\section{The no blow-up regime}\label{sec:no_blow}

In this section we prove that in the regime $\alpha \geq 0$, there are no infinite clusters for all finite times almost surely; see Theorem~\ref{thm:finite_clusters}. Note that the proof of the theorem holds for any stationary and ergodic starting configuration and a wide family of graphs and similar models. An alternative proof which only works for $\alpha\ge 1$, but gives some insights on cluster formation is available in Appendix~\ref{appendix}. 
\begin{definition}
    For any $x\in\BZ^d$, let $\mathfs{K}_x(t)$ be the cluster containing the point $x$. If no such cluster exists then $\mathfs{K}_x(t)=\emptyset$. 
\end{definition}
\begin{proof}[Proof of Theorem \ref{thm:finite_clusters}]
The proof is based on a density-flux argument in the spirit of Smoluchowski equations \cite{smoluchowski1918versuch}, but with scaling and smoothing. For simplifying notations we assume w.l.o.g. $p=1$. For \(k\ge 1\), define
\begin{equation}\label{def:Vmt}
    V_k(t)
    :=
    \lim_{N\to\infty}
    \frac{1}{N^d}
    \sum_{x\in [-N/2,N/2]^d}
    \ind_{\{|\mathfs{K}_x(t)|=k\}}.
\end{equation}
By translation invariance and ergodicity,
\[
    V_k(t)=\prob(|\mathfs{K}_0(t)|=k).
\]
Thus
\[
    \sum_{k\ge 1} V_k(t)
    =
    \prob(|\mathfs{K}_0(t)|<\infty).
\]

We now group cluster sizes into dyadic blocks, with a linear interpolation
near the endpoints. For \(n\ge 0\), define
\[
\begin{aligned}
    W_n(t)
    &:=
    \sum_{k=2^n+1}^{2^{n+1}} V_k(t)
    +
    \frac{8}{2^n}
    \sum_{k=\frac78 2^n}^{2^n}
    V_k(t)\left(k-\frac78 2^n\right)  
    -
    \frac{8}{2^{n+1}}
    \sum_{k=\frac78 2^{n+1}}^{2^{n+1}}
    V_k(t)\left(k-\frac78 2^{n+1}\right).
\end{aligned}
\]
(If the summation bounds are not integer, which 
can only happen if for $n < 3$, we may replace 
them by their integer value.)

In particular, we see that $W_n(t)\ge 0$ and define 
\[
 W(t):=  \sum_{n\ge 0}W_n(t)=\sum_{k\ge 1}V_k(t).
\]

We next estimate the total rate at which mass moves between dyadic scales. We say that a cluster is on the $m^{\textup{th}}$ dyadic scale if its size is in $[\frac{7}{8}2^{m},2^{m+1}]$, i.e. it contributes to $W_m(t)$. For fixed $x \in \Z^d$ and $t \geq 0$, let $A^{(t)}_{n}(x)$ denote the event that the cluster $\mathfs{K}_x(t)$ connects at time $t$ to a cluster whose size lies in the enlarged dyadic window $[2^{n-1},2^{n+1}]$.

Similarly, let \(a_n^{(m)}(t)\) denote the probability that a cluster whose
size lies in the \(m\)-th dyadic scale connects, at time \(t\), to a cluster
whose size lies $ [2^{n-1},2^{n+1}]$, conditional on the event that the cluster performs a jump at time $t$, i.e. 
\begin{equation}
a_n^{(m)}(t)=\E{\ind_{\{|\mathfs{K}_0(t-)|\in[\frac{7}{8}2^{m},2^{m+1}]\} \cap A^{(t)}_{n}(0)}\Big|\mathfs{K}_0\text{ rings at time }t}
.\end{equation}
Since a single connection event involves only one partner cluster, and since
the enlarged dyadic windows have uniformly bounded overlap, we have
$
    \sum_{n\ge 0} a_n^{(m)}(t)\le C_0
$
for an absolute constant \(C_0\), and with the specific choice of the windows above, one may simply take
\(C_0=3\).

A cluster in the $m^{\textup{th}}$ dyadic scale has size of order \(2^m\), and hence
jump rate at most of order
$
    2^{-m\alpha}
$
with  \(\alpha\ge 0\). The density of such clusters is of order
\(W_m(t)/2^m\). If such a cluster connects to a cluster in the $n^{\textup{th}}$
dyadic window, then the total possible increase in the smoothed weight
\(W_n\) is bounded by a constant times \(2^m\). The linear interpolation in
the definition of \(W_n\) ensures that this bound is uniform even when the
new size lies near the boundary between two dyadic blocks. More explicitly,
the gain is bounded by
\[
    2^m+\frac{8}{2^n}\,2^m\,2^n
    \le 9\cdot 2^m.
\]
Therefore the positive contribution to \(d W_n(t)/d t\) coming from clusters in
scale \(m\) is bounded by 
\[
    9\,W_m(t)\,2^{-m\alpha}\,a_n^{(m)}(t).
\]

Summing over \(n\) and \(m\), we obtain
\[
\begin{aligned}
    \frac{d}{dt}\mathcal W(t)
    &\le
    \sum_{n\ge 0}
    \left|\frac{d}{dt}W_n(t)\right| \\
    &\le
    9\sum_{m\ge 0}\sum_{n\ge 0}
    W_m(t)\,2^{-m\alpha}\,a_n^{(m)}(t) \\
    &\le
    9C_0\sum_{m\ge 0} W_m(t)2^{-m\alpha}.
\end{aligned}
\]
Since \(\alpha\ge 0\),
\[
    2^{-m\alpha}\le 1.
\]
Hence
\[
    \frac{d}{dt}\mathcal W(t)
    \le
    9C_0\,\mathcal W(t).
\]
In particular, for every finite \(T\),
\[
    \int_0^T
    \sum_{m\ge 0}\sum_{n\ge 0}
    W_m(t)2^{-m\alpha}a_n^{(m)}(t)\,dt
    <\infty.
\]

We now introduce the integrated flux between dyadic scales. Let
\(\alpha_{m,n}(t)\) denote the rate at which vertex-density mass moves from
scale \(m\) to scale \(n\). From the estimate above,
\[
    \alpha_{m,n}(t)
    \le
    9\,W_m(t)\,2^{-m\alpha}\,a_n^{(m)}(t).
\]
Define
\[
    \Theta_{m,n}(T):=\int_0^T \alpha_{m,n}(t)\,dt.
\]
Then
\[
\begin{aligned}
    \sum_{m\ge 0}\sum_{n\ge 0}\Theta_{m,n}(T)
    &\le
    9\int_0^T
    \sum_{m\ge 0}\sum_{n\ge 0}
    W_m(t)2^{-m\alpha}a_n^{(m)}(t)\,dt  \\
    &\le
    9C_0\int_0^T \sum_{m\ge 0}W_m(t)2^{-m\alpha}\,dt  \\
    &\le
    9C_0T.
\end{aligned}
\]
Consequently,
for every \(\varepsilon>0\), there exists \(N=N(\varepsilon,T)\) such that
\[
    \sum_{n>N}\sum_{m\le n}\Theta_{m,n}(T)<\varepsilon.
\]

We now control the mass that has left the first \(N\) dyadic scales by time
\(T\). Since mergers of clusters whose sum of volumes is in scale smaller than $N$ preserves 
$
    \sum_{n=0}^N W_n(T)
$
, in order to study 
$
   \sum_{n > N} W_n(T)
$
it suffices to consider the clusters crossing from a scale \(m\le N\) to a scale \(n>N\). Therefore, by the Fundamental Theorem of Calculus for generalized derivatives
(Dini derivatives),
\[
    \frac{d}{dt}\sum_{n=0}^N W_n(t)
    \ge
    -\sum_{m\le N<n}\alpha_{m,n}(t).
\]
Integrating from \(0\) to \(T\), and using
$
    \sum_{n\ge 0}W_n(0)=1,
$
we obtain
\[
\begin{aligned}
    1-\sum_{n=0}^N W_n(T)
    &\le
    \sum_{m\le N<n}\Theta_{m,n}(T)  \\
    &\le
    \sum_{n>N}\sum_{m\le n}\Theta_{m,n}(T)  \\
    &<\varepsilon.
\end{aligned}
\]
Since \(\varepsilon>0\) was arbitrary,
$
    \sum_{n\ge 0}W_n(T)=\sum_{k\ge 1}V_k(T)=1.
$
Thus
$
    \prob(|\mathfs{K}_0(T)|<\infty)=1.
$
Since \(T<\infty\) was arbitrary, this proves the theorem.
\end{proof}

As an immediate consequence of the above arguments, we  have the following bounds  on the cluster density and the expected size of the cluster containing the origin.

The following corollary is the first step into the quantitative estimation of the sizes of the clusters in our model. This topic is covered in our next paper \cite{withborgnia}.
\begin{corollary}\label{cor:ClusterBound} Let $p=1$ and $\alpha>0$. Write $f(t)$ for the cluster density, i.e.,  
\begin{equation*}
    f(t):=   \sum_{m \geq 1} \frac{V_m(t)}{m}  = \BE\left[ \frac{1}{\mathfs{K}_0(t)}\right]
\end{equation*}
with $V_m(t)$ from \eqref{def:Vmt}. Then there exist constants $C_1,C_2,C_3>0$ such that
\begin{equation}\label{eq:ExpectedClusterSize1}
    f(t)\le (C_1 t +C_2)^{-1/\alpha}
\end{equation} for all $t>0$. Moreover, we get that
\begin{equation}\label{eq:ExpectedClusterSize2}
    \BE\left[|\mathfs{K}_0(t)|\right]\ge\frac{1}{f(t)}\ge C_3 t^{\frac{1}{\alpha}} , 
\end{equation}
and that for every $\epsilon$ there exists $c > 0$ such that for all $t$ large enough
\begin{equation}\label{eq:ClusterSize2quantilesbambam}
    \BP\left[|\mathfs{K}_0(t)|\ge C_3 t^{\frac{1}{\alpha}}\right] \geq 1 - \epsilon. 
\end{equation}

\end{corollary}
\begin{proof}
Note that $\sum_{m \geq 1} V_m(t)$ is conserved in $t$, and that
\begin{equation}\label{eq:inversenumclusters}
    \frac{d}{dt}f(t)
    =
    -\sum_{m\ge1} V_m(t)m^{-\alpha-1} .
\end{equation}
Equation \eqref{eq:inversenumclusters} follows from the fact that $f(t)$ is the total density of clusters, and $V_m(t)m^{-\alpha-1}$ is the rate at which clusters of size $m$ merge into other clusters.

Define a probability measure on cluster sizes by
\begin{equation}
     \mu_t(m):= \frac{V_m(t)}{m f(t)} .
\end{equation}
and apply Jensen's inequality to get
\[
\begin{aligned}
    \sum_{m\ge1}\frac{V_m(t)}{m f(t)}m^{-\alpha}
    \ge
    \left(
        \sum_{m\ge1}\frac{V_m(t)}{f(t)}
    \right)^{-\alpha} =f(t)^\alpha .
\end{aligned}
\]
Multiplying by \(f(t)\), this results in the differential inequality 
\[
 \frac{d}{dt}f(t)
    \le
    -p^{-\alpha} f(t)^{\alpha+1}, 
\]
which in return yields \eqref{eq:ExpectedClusterSize1} as desired. Another application of Jensen gives  the second statement~\eqref{eq:ExpectedClusterSize2}. \eqref{eq:ClusterSize2quantilesbambam} follows from Markov's inequality applies to \eqref{eq:ExpectedClusterSize1}.
\end{proof}

\

\section{The blow-up regime}\label{sec:blowup}
In this section we prove that in sufficiently negative $\alpha$ the Cluster-Cluster model has a blowup. We start with the fully packed case. 

\subsection{Blowup for $p=1$}
\begin{proof}[Proof of Theorem \ref{thm:blowupregieme} (1)]
Fix the cluster $(\mathfs{C}_0(t))_{t \geq 0}$ containing the origin at time $0$. Since $p=1$, every site is occupied at time $0$, and thus every attempted move is blocked by another cluster.
Let $T_i$ denote the waiting time between the $(i-1)$-st and $i$-th clock ring, and note that
\begin{equation}
    |\mathfs{C}_{0}(S_n)| \geq n
\end{equation} where $S_n := \sum_{i=1}^{n} T_{i}$ for all $n \in \BN$. Since a cluster of size at least $n$ rings
at rate at least $n^{-\alpha}$, we may couple the variables $T_n$ with
independent random variables
\begin{equation*}
    X_n \sim \operatorname{Exp}\left(n^{-\alpha}\right)
\end{equation*}
in such a way that
\begin{equation}\label{eq:CouplingBound}
    T_n \leq X_n
\end{equation}
for every $n \geq 1$. Set $ \beta
    = \frac{\alpha-1}{2}$, and define
\begin{equation*}
    A_n := \left\{X_n \geq n^\beta\right\}.
\end{equation*}
Using the exponential tails of $T_n$, and the assumption $\alpha<-1$,  it follows that
\begin{equation*}
    \sum_{n=1}^{\infty}\mathbb{P}(A_n)<\infty.
\end{equation*}
Hence, by the first Borel--Cantelli lemma, almost surely only finitely many
of the events $A_n$ occur, and almost surely, there exists a random
$N<\infty$ such that
\begin{equation*}
    X_n<n^\beta
\end{equation*}
for every $n\geq N$. Since $\beta<-1$, we obtain from \eqref{eq:CouplingBound} that
\begin{equation}\label{eq:BoundedXns}
    \sum_{n=1}^{\infty}T_n
    \leq
    \sum_{n=1}^{\infty}X_n
    <\infty
\end{equation}
almost surely. Hence the cluster undergoes infinitely many mergers in finite
time. Now to show that such an explosion may occur in an arbitrarily short time with positive probability, fix $\varepsilon>0$, and note that from \eqref{eq:BoundedXns}, we see that for all sufficiently large
$N$,
\begin{equation*}
    \mathbb{P}\left(
        \sum_{n=N}^{\infty}X_n<\frac{\varepsilon}{2}
    \right)>0.
\end{equation*}
Moreover, for every fixed $N$,
\begin{equation*}
    \mathbb{P}\left(
        \sum_{n=1}^{N-1}X_n<\frac{\varepsilon}{2}
    \right)>0,
\end{equation*}
since the random variables $X_1,\ldots,X_{N-1}$ have strictly positive
densities on $(0,\infty)$. Since the two sums depend on disjoint families
of independent random variables,
\begin{equation}\label{eq:PositiveExplosion}
\begin{aligned}
  \mathbb{P}\left(
        \sum_{n=1}^{\infty}T_n<\varepsilon
    \right) \geq   \mathbb{P}\left(
        \sum_{n=1}^{\infty}X_n<\varepsilon
    \right)
    \geq
    \mathbb{P}\left(
        \sum_{n=1}^{N-1}X_n<\frac{\varepsilon}{2}
    \right)
    \mathbb{P}\left(
        \sum_{n=N}^{\infty}X_n<\frac{\varepsilon}{2}
    \right) >0.
\end{aligned}
\end{equation}
Thus the cluster started at $0$ explodes before time $\varepsilon$ with positive probability. Finally, since the event that some cluster explodes before time $\varepsilon$ is
translation invariant, it has probability either $0$ or
$1$, and thus, in view of \eqref{eq:PositiveExplosion}, must occur almost surely. Since $\varepsilon>0$ was arbitrary, the process
exhibits an immediate blowup almost surely, allowing us to conclude.
\end{proof}

\begin{proof}[Proof of Theorem~\ref{thm:blowupregieme} (2)]
As before in the proof of Theorem~\ref{thm:blowupregieme} (1), since $p=1$, all clusters cannot move. We now give an equivalent graphical construction of the dynamics in terms of Bernoulli bond percolation. Assign
to every site $x\in\BZ^d$ an independent rate $1$ Poisson clock. At each ring of the clock at site $x$, choose a unit direction
\begin{equation*}
    r\in\cb{\pm \textup{e}_i:i=1,\ldots,d}
\end{equation*}
uniformly at random. For a cluster $C$, the union of the clocks attached to the vertices of $C$ is a Poisson clock of rate $|C|$. Since $\alpha=-1$, this is exactly the prescribed clock rate $|C|$.
Moreover, conditional on a ring of the clock of $C$, the vertex at which
the ring occurs is uniformly distributed in $C$.
Fix the direction $r$, and let
\begin{equation*}
    \partial_r C
    :=
    \cb{\cb{z,z+r}:z\in C,\ z+r\notin C}
\end{equation*}
denote the set of edges in direction $r$ joining $C$ to another cluster.
Suppose that the clock at $x\in C$ rings. Now choose
an edge uniformly from $\partial_r C$ and add it. This produces precisely
the correct transition rule as any
fixed edge in $\partial_r C$ is selected with probability $(|\partial_r C|)^{-1}$, and any blocking edge is contained in at most one such boundary set $\partial_r C$.

We now couple the process to Bernoulli bond percolation. For every
unoriented edge $e=\cb{x,y}\in E(\BZ^d)$, declare $e$ open in the
percolation process when either the clock at $x$ rings and chooses the
direction $y-x$, or the clock at $y$ rings and chooses the direction
$x-y$. Each of these two events occurs at rate $1/(2d)$, independently.
Consequently, the opening time of each edge is an independent exponential with rate $1/d$. Hence,
at time $t$, the resulting edge configuration is Bernoulli bond
percolation with parameter
\begin{equation*}
    q(t)=1-e^{-t/d}.
\end{equation*}
We claim that, under this coupling, the endpoints of every open
percolation edge belong to the same cluster in the Cluster-Cluster model.
Indeed, suppose that the percolation clock of $e=\cb{x,y}$ rings. If $x$
and $y$ already belong to the same Cluster-Cluster cluster, there is
nothing to prove. Otherwise, say that the clock at $x$ rings and chooses
the direction $y-x$. Then $e$ is a blocking edge for the cluster containing
$x$, and the graphical construction above adds $e$. Thus the two clusters
containing $x$ and $y$ merge. The claim follows inductively over the
percolation opening times. 
It follows that every connected component of the Bernoulli percolation
configuration at time $t_0$ corresponds to vertices of some cluster of the
Cluster-Cluster model at time $t_0$. Now we may choose $t_0<\infty$ sufficiently large that
\begin{equation*}
    q(t_0)>p_c(\BZ^d) ,
\end{equation*} where $p_c(\BZ^d)$ is the critical value for Bernoulli bond percolation on $\BZ^d$ with $d>1$. 
At time $t_0$, the Bernoulli percolation configuration therefore contains
an infinite connected component almost surely, and so does the Cluster-Cluster model by the above coupling. Thus the process blows up in finite time almost surely.  This proves that the Cluster-Cluster model for $p=1$, $\alpha=-1$ and $d>1$ exhibits almost surely a blowup in finite time. 
\end{proof}

\section{Blowup for $p<1$ and $\alpha < -1 - 2/d$}\label{sec:vlowp<1}

In order to show blowup for $p<1$ for sufficiently negative $\alpha$, we require some setup. Throughout, write
$$m_n:=2^n,\qquad I_n:=[m_n,2m_n).$$ We use the term scale-n clusters for clusters with their size in $I_n$. 
We have the following procedure to color all vertices as either red or black. Initially, color every occupied particle black. Once a particle is recolored red, it remains red forever. Let $L_x(t)\in\{0,1\}$ denote the color of the particle at $x$ at time $t$, where $0$ means black and $1$ means red. We explain below the precise recoloring mechanism. Define
$$B_n(t):=\BP\left(L_0(t)=0,\ |\mathfs{K}_0(t)|\in I_n\right).$$
Now for contradiction, assume that the Cluster-Cluster model with parameters $\alpha < -1 - 2/d$ and $p \in (0,1)$ is well-defined for all finite times. By the pointwise ergodic theorem for $\BZ^d$-actions, for any $t \geq 0$ and $n\in \BN$,
$$B_n(t)=\lim_{N\to\infty}\frac{1}{N^d}\sum_{x\in[0,N-1]^d}\mathds{1}_{\left\{L_x(t)=0,\ |\mathfs{K}_x(t)|\in I_n\right\}} \text{ a.s.}$$
Let $a_n= c p n^{-2}$ for a sufficiently small constant $c>0$ so that $\sum_{n=1}^{\infty} a_n \leq \frac{p}{2}$. Define the stopping times
$$\mathcal T_0:=0,\qquad\mathcal T_n:=\inf\left\{t>\mathcal T_{n-1} : B_n(t)<a_n\right\}.$$
At time $\mathcal T_n$, recolor red every particle belonging to a cluster of size less than $2m_n$. Hence the black density $B_n(t)$ is non-increasing on $[\mathcal T_{n-1},\infty)$. Our goal is to prove that $\mathcal T:=\lim_{n\to\infty}\mathcal T_n<\infty$.

\begin{lemma}
\label{lem:distance}
There is a constant $C=C(d)<\infty$ such that the following holds: Fix $n$ and $t\in[\mathcal T_{n-1},\mathcal T_n)$, and let
$$r_n:=C\left(\frac{m_n}{a_n}\right)^{1/d}.$$
Then at least half of the black scale-$n$ particle density belongs to clusters having another black scale-$n$ cluster within graph distance $r_n$.
\end{lemma}
\begin{proof}
We call a black scale-$n$ cluster isolated if it has no other black scale-$n$ cluster within graph distance $r_n$. For each isolated cluster $\mathcal{C}$, define $\mathcal N(\mathcal{C}):=\left\{x\in\BZ^d:\operatorname{dist}(x,\mathcal{C})\leq r_n/3\right\}$. These sets, over isolated black scale-$n$ clusters, are disjoint. Moreover, $|\mathcal N(\mathcal{C})|\geq c_d r_n^d$. Since each cluster has fewer than $2m_n$ vertices, the density of black particles in isolated scale-$n$ clusters is at most $2m_n c_d^{-1} r_n^{-d}.$ Choosing $C$ sufficiently large makes this quantity at most $a_n/2$. Since $t<\mathcal T_n$, we have $B_n(t)\geq a_n$, so at most half of the black scale-$n$ density is isolated.
\end{proof}

Next, we bound the escape rate of clusters in the Cluster-Cluster model. Let $\operatorname{cap}(A)$ denote
 the discrete capacity of a finite set $A\subset\BZ^d$ with $d \geq 3$. The following lemma is a direct
 consequence of the Green function representation of hitting probabilities; see, for example, \cite[Sections~4.3 and~6.5]{LawlerLimic2010}, so we omit the proof.

\begin{lemma}
\label{lem:escape}
For every $d\geq 3$, there exists $c=c(d)>0$ such that, for every finite $A\subseteq B(z,r)$ and every $x\in B(z,2r)$, and all $r \ge 1$,
$$\BP_x\left(\tau_A<\tau_{B(z,4r)^c}\right)\geq c\frac{\operatorname{cap}(A)}{r^{d-2}}.$$
\end{lemma}

With this bound at hand, we provide an upper bound on the time it takes for a cluster to leave the scale $I_n$.

\begin{lemma}
\label{lem:hitting_rigorous}
Let $d\ge2$. There exist constants $A,q>0$ and $\kappa<\infty$, depending only on $d,p,\alpha$, such that the following holds. Let $t\in[\mathcal T_{n-1},\mathcal T_n)$, and let $\mathcal{C}$ be a black cluster with $|\mathcal{C}|\in I_n$. Assume there exists another black cluster $\mathcal{D}\neq\mathcal{C}$ satisfying $|\mathcal{D}|\in I_n$ and $\operatorname{dist}(\mathcal{C},\mathcal{D})\le r_n$. Then, conditional on the configuration at time $t$, $\mathcal{C}$ exits $I_n$ within time $h_n = A m_n^{1+\alpha} r_n^2 n^\kappa$ with probability at least $q n^{-\kappa}$.
\end{lemma}
\begin{proof}
We count attempted moves of $\mathcal{C}$. While $\mathcal{C}\in I_n$, the ringing rate is bounded below by $m_n^{-\alpha}$. If an attempted move is blocked, the cluster merges, and its size increases by at least one. Thus, $\mathcal{C}$ can experience at most $m_n$ blocked attempts before its mass reaches $2m_n$, forcing it to exit $I_n$. 

To completely avoid correlations caused by the environment selectively blocking directions, we employ a pigeonhole argument. Let $M_n = C r_n^2 n^\kappa$ be the number of steps required for a pure simple random walk (SRW) to traverse distance $4r_n$ with 
probability at least $1-\exp(-c n^{\kappa})$ for some constants $c,C>0$. We allocate a total of $L_n$ attempts, partitioned into $m_n + 1$ contiguous blocks of length $M_n$:
$$L_n := (m_n + 1) M_n.$$
Since there can be at most $m_n$ blocked moves total, at most $m_n$ of these blocks can contain a blockage. By the pigeonhole principle, if $\mathcal{C}$ has not exited $I_n$, there must be at least one block of length $M_n$ wherein \textit{every single attempted move is unblocked}.

During this specific block, the cluster's translation exactly couples with an independent SRW of length $M_n$, completely free of any adversarial drift or environment-induced correlation. At the start of this block, the distance to $\mathcal{D}$ is at most $r_n$ (since original vertices do not retreat during mergers). 

For $d\ge3$, by Lemma \ref{lem:escape} and standard heat kernel estimates, the probability this pure SRW hits $\mathcal{D}$ before escaping distance $4r_n$ is bounded below by
$$c\frac{\operatorname{cap}(\mathcal{D})}{r_n^{d-2}} \ge c_1\frac{m_n^{1-2/d}}{r_n^{d-2}} = c_2 a_n^{(d-2)/d} \ge c_3 n^{-2(d-2)/d}.$$
For $d=2$, standard estimates for SRW hitting a set of size $\ge m_n$ within distance $r_n$ yield a lower bound of $q n^{-\kappa}$ for some constant $q>0$.

Whenever this exact SRW trajectory hits $\mathcal{D}$, the clusters merge, and since $|\mathcal{D}| \ge m_n$, $\mathcal{C}$ instantly exits $I_n$. The time to realize $L_n$ attempted moves is dominated by the sum of $L_n$ exponential variables of rate $m_n^{-\alpha}$. A Chernoff bound ensures this sum is at most $h_n := A m_n^\alpha L_n \le A' m_n^{1+\alpha} r_n^2 n^\kappa$ with probability tending to $1$ as $n \rightarrow \infty$.
\end{proof}

\begin{proof}[Proof of Theorem~\ref{thm:blowupregieme}(3)]
Fix $n$ and $t\in[\mathcal T_{n-1},\mathcal T_n)$. By Lemma \ref{lem:distance}, at least $B_n(t)/2$ of the scale-$n$ density lies in clusters satisfying Lemma \ref{lem:hitting_rigorous}. Iterating the hitting bound $j \leq C n^\kappa\log n$ times forces the density below $a_n$. Hence
$$\mathcal T_n-\mathcal T_{n-1} \leq C n^\kappa(\log n)h_n,$$
where $h_n = A m_n^{1+\alpha} r_n^2 n^\kappa$. Using $r_n^d \asymp m_n / a_n = m_n n^2 / (cp)$, we have $r_n^2 \le C' m_n^{2/d} n^{4/d}$. Therefore,
$$h_n \leq C'' m_n^{1+2/d+\alpha} n^{\kappa+4/d}.$$
It follows that
$$\mathcal T_n-\mathcal T_{n-1} \leq C 2^{n(1+2/d+\alpha)} n^{2\kappa+4/d}(\log n).$$
Because we assume $\alpha < -1 - 2/d$, the exponent satisfies $1+2/d+\alpha < 0$. This guarantees the right-hand side decays exponentially in $n$, implying the limit $\mathcal T:=\lim_{n\to\infty}\mathcal T_n$ is finite. At time $\mathcal T$, all mass in finite clusters has been recolored red (at most $p/2$ by our assumptions). Since the  total density is $p$, an infinite cluster must hold the remaining $p/2$ density, proving finite-time blowup almost surely.
\end{proof}
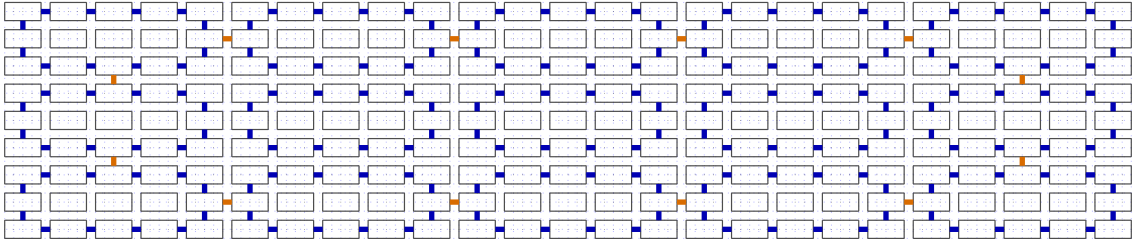
\begin{figure}[H]
\centering
\begin{tikzpicture}[xscale=0.12,yscale=0.12, rotate=90,
    vtx/.style={circle,fill=black,inner sep=0.55pt},
    levzero/.style={black!70,line width=0.45pt},
    levone/.style={blue!70!black,line width=2pt},
    levtwo/.style={orange!85!black,line width=2pt},
    gridstyle/.style={blue!25,thin,dotted}
]


\draw[step=1cm,gridstyle] (-13,-62) grid (13,62);

\foreach \x in {-13,...,13}{%
    \foreach \y in {-62,...,62}{%
        \fill[black] (\x,\y) circle (1pt);
    }%
}

\foreach \y in {-62,-58,-57,-53,-52,-48,-47,-43,-42,-38,-37,-33,-32,-28,-27,-23,-22,-18,-17,-13,-12,-8,-7,-3,-2,2,3,7,8,12,13,17,18,22,23,27,28,32,33,37,38,42,43,47,48,52,53,57,58,62}{%
    \foreach \x in {-13,-10,-7,-4,-1,2,5,8,11}{%
        \draw[levzero] (\x,\y) -- ({\x+2},\y);
    }%
}

\foreach \x in {-13,-11,-10,-8,-7,-5,-4,-2,-1,1,2,4,5,7,8,10,11,13}{%
    \foreach \a/\b in {-62/-58,-57/-53,-52/-48,-47/-43,-42/-38,-37/-33,-32/-28,-27/-23,-22/-18,-17/-13,-12/-8,-7/-3,-2/2,3/7,8/12,13/17,18/22,23/27,28/32,33/37,38/42,43/47,48/52,53/57,58/62}{%
        \draw[levzero] (\x,\a) -- (\x,\b);
    }%
}

\foreach \y in {-60,-40,-35,-15,-10,10,15,35,40,60}{%
    \foreach \x in {-11,-8,-2,1,7,10}{%
        \draw[levone] (\x,\y) -- ({\x+1},\y);
    }%
}

\foreach \x in {-12,-6,-3,3,6,12}{%
    \foreach \a/\b in {-58/-57,-53/-52,-48/-47,-43/-42,-33/-32,-28/-27,-23/-22,-18/-17,-8/-7,-3/-2,2/3,7/8,17/18,22/23,27/28,32/33,42/43,47/48,52/53,57/58}{%
        \draw[levone] (\x,\a) -- (\x,\b);
    }%
}

\foreach \y in {-50,50}{%
    \foreach \x in {-5,4}{%
        \draw[levtwo] (\x,\y) -- ({\x+1},\y);
    }%
}

\foreach \x in {-9,9}{%
    \foreach \a/\b in {-38/-37,-13/-12,12/13,37/38}{%
        \draw[levtwo] (\x,\a) -- (\x,\b);
    }%
}

\end{tikzpicture}
\caption{Renormalization step. \textcolor{blue}{blue: level 1 connections} and \textcolor{orange}{orange: Level 2 connections}}
\label{fig:renorm}
\end{figure}

\section{Special starting configurations}\label{sec:startconfig}
In the following, we construct special starting configurations for the Cluster-Cluster model, which illustrate the sensitivity of the occurrence of blowups in the regime $\alpha \in (-1,0)$. Let us stress that we construct the starting configurations as subgraphs of $(\BZ^d,\mathcal{E}(\BZ^d))$, without fulfilling the requirements of Definition~\ref{def:clustercluster}. In order to obtain proper starting configurations for the Cluster-Cluster model, one simply takes a uniform spanning forest of the respective configurations. 

We start with building for any $\alpha\in(-1,0)$ an example which immediately blows up.
\begin{proof}[Proof of Theorem \ref{bad_start_config}]
    First consider the following starting configuration with $p=1$, constructed via applying successive renormalization of the square lattice. 
    
    The renormalization step with parameter $\ell$ and dimension $d \ge 2$ is as follows. 
    We start by tiling the lattice with $3 \times \ell \times \ell \times ... \times \ell$ boxes made of nodes.
    To preserve translation invariance, we create the tiling by first creating the box containing $0$ by choosing a possible box uniformly out of all the possible boxes containing $0$. 
    For each box, we connect the exterior nodes to one cluster, and keep the nodes inside the boxes disconnected.
    Each box will now act as a single node in the renormalized lattice, where two nodes are connected via the center edge between the two.

    We apply the renormalization step to $\BZ^d$ $n$ times and take $n \to \infty$.     
    An example of the first $3$ renormalization steps with $d=2,\ell=5$ can be seen in Figure \ref{fig:renorm}.

    Note that for each vertex $x\in \BZ^d$ it takes finitely many steps almost surely to decide its final cluster.
    This is because at each step, there is $\frac{(\ell-2)^{d-1}}{3\ell^{d-1}}\ge \frac{1}{3^d}$ probability to have it be in the interior of a box, after which it makes no more connections.
    We use this to assign a layer to each cluster - the layer of the cluster is the number of renormalizations it has been affected by.

    We can now bound from below the rate at which a cluster in layer $k$ connects to a cluster in layer $k+1$.
    
    Denoting by $s_k$ the size of a cluster in layer $k$, we can see that $s_k \ge \ell^{d-1} s_{k-1}$ if we only consider the clusters of layer $k-1$ on the largest face of the box in renormalization step $k$.
    This gives $s_k \ge \rb{\ell^{d-1}}^k$

    Now consider just the rate of connection via one of the two largest faces (with dimensions $\ell \times \ell \times ... \times \ell$). 
    A cluster attempts to move to one of these two directions with probability $\frac{1}{d}$. When it attempts to move, it will connect to the higher layer with probability $\frac{G_k}{G_k+B_k}$, where $G_k$ is the number of edges between the cluster in layer $k+1$ and the cluster in layer $k$, and $B_k$ is the number of edges to other clusters in smaller layers.
    We can see that $G_k = \rb{\ell^{d-1}}^k$ as any edge outside the layer $k$ box will connect to layer $k+1$. For $B_k$, we can see that connections to smaller clusters can only happen inside the box of layer $k$, so by the volume of the box we get $B_k \le \rb{3\ell^{d-1}}^k$
    
    Combining the above gives an effective rate of increasing from layer $k$ to layer $k+1$ of
    \begin{equation}
        \begin{aligned}
            r_{\textup{eff}}^{(k)} \ge \frac{1}{d}s_k^{-\alpha}\frac{G_k}{G_k+B_k} 
            &\ge \frac{1}{d}\rb{\frac{\rb{\ell^{d-1}}^k\rb{\rb{\ell^{d-1}}^{-\alpha}}^k}{3^k\rb{\ell^{d-1}}^k+\rb{\ell^{d-1}}^k}} \ge \frac{1}{2d}\rb{\frac{\rb{\ell^{d-1}}^{-\alpha}}{3}}^k.
        \end{aligned}
    \end{equation}
    If we now take $\ell$ large such that $\frac{\rb{\ell^{d-1}}^{-\alpha}}{3} > 2$, we get $r_{\textup{eff}}^{(k)} \ge 2^k$ for large enough $k$ (recalling $\alpha<0$). Here, note that $\ell \rightarrow \infty$ as $\alpha \rightarrow 0$. We can see by Borel-Cantelli that this is sufficient. More precisely, let $\ep>0$, and note that we can stochastically dominate the rate to go from layer $k$ to layer $k+1$ by the law of random variables $X_i$, where $X_i \sim \operatorname{Exp}(2^k), p_i = \Prob{X_i > \ep} = e^{-\ep2^k}$. Since $\sum_{i=1}^{\infty} p_i < \infty$, we have that almost surely there exists $k$ such that all layers connect to the one above in $\ep$ time, i.e., an infinite cluster in $\ep$ time.

    To obtain density $p\in(0,1)$, set
\[
q:=\frac{(\ell-2)^{d-1}}{3\ell^{d-1}},
\qquad
\gamma:=1-q.
\]
The density of vertices of layer $k$ is $q\gamma^k$, while the density of
vertices of layer at least $k$ is $\gamma^k$. Choose $K\ge0$ such that
\[
\gamma^{K+1}\le p\le\gamma^K,
\]
delete all clusters of layers smaller than $K$, and independently delete
each layer-$K$ cluster with probability
\[
r:=\frac{\gamma^K-p}{q\gamma^K}\in[0,1].
\]
The density of the remaining configuration is
\[
\gamma^{K+1}+(1-r)q\gamma^K=p.
\]
All clusters of layers larger than $K$ remain, so the infinite
hierarchical tails are unchanged.
\end{proof}


Conversely, we now show that for any $\alpha\in(-1,0)$ there is an initial configuration for which the Cluster-Cluster model doesn't blow up.

\begin{proof}[Proof of Theorem \ref{thm:good_start_config}]
    We use a multi-scale renormalization argument. The base Level 0 is a $3 \times \dots \times 3$ cube with its center vertex disconnected; see Figure~\ref{fig:BaseCase}.

    \begin{figure}[h]
    \centering
    \begin{tikzpicture}[scale=1]
    \def\a{0.22}
    \def\gap{0.85}
    \begin{scope}[shift={(-2.2,3.2)}, scale=\a]
        \foreach \x/\y in {0/0,1/0,2/0,0/1,2/1,0/2,1/2,2/2}{
            \fill (\x,\y) circle (2.2pt);
        }
        \fill[red] (1,1) circle (2.8pt);
        \draw (0,0)--(1,0)--(2,0);
        \draw (0,2)--(1,2)--(2,2);
        \draw (0,0)--(0,1)--(0,2);
        \draw (2,0)--(2,1)--(2,2);
    \end{scope}
    \end{tikzpicture}
    \caption{\label{fig:BaseCase}Base level $0$ configuration with a disconnected center.}
    \end{figure}
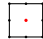

    At step $k \ge 1$, we construct a layer-$k$ macro-cube of side length $3N_k n$ (where $n = n_{k-1}$ is the layer-$(k-1)$ sub-cube side length) by arranging a grid of $N_k^d$ layer-$(k-1)$ sub-cubes; see Figure~\ref{fig:Hierachie}. In a similar way to Theorem \ref{thm:good_start_config}, we make the configuration translation invariant by choosing 0 to be in a uniform box in each level. We aim to choose $N_k$ large enough so that the probability of layer $k$ connecting to layer $k+1$ before time $k$ is at most $\frac{1}{k^2}$.

    \begin{figure}[h]
    \centering
    \begin{tikzpicture}[scale=0.92]
    \def\a{0.22}
    \def\gap{0.85}
    \fill[gray!20] ({3*\gap-0.12},{3*\gap-0.12}) rectangle ({5*\gap+2*\a+0.12},{5*\gap+2*\a+0.12});
    \foreach \i in {0,...,8}{
        \foreach \j in {0,...,8}{
            \begin{scope}[shift={({\i*\gap},{\j*\gap})}, scale=\a]
                \foreach \x/\y in {0/0,1/0,2/0,0/1,2/1,0/2,1/2,2/2}{
                    \fill (\x,\y) circle (2.2pt);
                }
                \fill[red] (1,1) circle (2.8pt);
                \draw (0,0)--(1,0)--(2,0);
                \draw (0,2)--(1,2)--(2,2);
                \draw (0,0)--(0,1)--(0,2);
                \draw (2,0)--(2,1)--(2,2);
            \end{scope}
        }
    }
    \foreach \j in {0,1,2,6,7,8}{
        \foreach \i in {0,...,7}{
            \draw[blue,dashed] ({\i*\gap+2*\a},{\j*\gap+\a}) -- ({(\i+1)*\gap},{\j*\gap+\a});
        }
    }
    \foreach \j in {3,4,5}{
        \foreach \i in {0,1,6,7}{
            \draw[blue,dashed] ({\i*\gap+2*\a},{\j*\gap+\a}) -- ({(\i+1)*\gap},{\j*\gap+\a});
        }
    }
    \foreach \i in {0,1,2,6,7,8}{
        \foreach \j in {0,...,7}{
            \draw[blue,dashed] ({\i*\gap+\a},{\j*\gap+2*\a}) -- ({\i*\gap+\a},{(\j+1)*\gap});
        }
    }
    \foreach \i in {3,4,5}{
        \foreach \j in {0,1,6,7}{
            \draw[blue,dashed] ({\i*\gap+\a},{\j*\gap+2*\a}) -- ({\i*\gap+\a},{(\j+1)*\gap});
        }
    }
    \end{tikzpicture}\hspace{1cm}
    \begin{tikzpicture}[scale=0.49]
    \def\B{1.15}      
    \def\gap{1.55}    
    \def\inner{0.38}  
    \foreach \i in {0,...,8}{
        \foreach \j in {0,...,8}{
            \pgfmathsetmacro{\x}{\i*\gap}
            \pgfmathsetmacro{\y}{\j*\gap}
            \draw[black, line width=0.45pt] (\x,\y) rectangle ({\x+\B},{\y+\B});
            \fill[gray!25] ({\x+0.5*\B-0.5*\inner},{\y+0.5*\B-0.5*\inner}) rectangle ({\x+0.5*\B+0.5*\inner},{\y+0.5*\B+0.5*\inner});
            \fill[red] ({\x+0.5*\B},{\y+0.5*\B}) circle (1.2pt);
        }
    }
    \fill[gray!12] ({3*\gap-0.18},{3*\gap-0.18}) rectangle ({5*\gap+\B+0.18},{5*\gap+\B+0.18});
    \foreach \i in {0,...,8}{
        \foreach \j in {0,...,8}{
            \pgfmathsetmacro{\x}{\i*\gap}
            \pgfmathsetmacro{\y}{\j*\gap}
            \draw[black, line width=0.45pt] (\x,\y) rectangle ({\x+\B},{\y+\B});
            \fill[gray!25] ({\x+0.5*\B-0.5*\inner},{\y+0.5*\B-0.5*\inner}) rectangle ({\x+0.5*\B+0.5*\inner},{\y+0.5*\B+0.5*\inner});
            \fill[red] ({\x+0.5*\B},{\y+0.5*\B}) circle (1.2pt);
        }
    }
    \foreach \j in {0,1,2,6,7,8}{
        \foreach \i in {0,...,7}{
            \draw[blue,thick] ({\i*\gap+\B},{\j*\gap+0.5*\B}) -- ({(\i+1)*\gap},{\j*\gap+0.5*\B});
        }
    }
    \foreach \j in {3,4,5}{
        \foreach \i in {0,1,6,7}{
            \draw[blue,thick] ({\i*\gap+\B},{\j*\gap+0.5*\B}) -- ({(\i+1)*\gap},{\j*\gap+0.5*\B});
        }
    }
    \foreach \i in {0,1,2,6,7,8}{
        \foreach \j in {0,...,7}{
            \draw[blue,thick] ({\i*\gap+0.5*\B},{\j*\gap+\B}) -- ({\i*\gap+0.5*\B},{(\j+1)*\gap});
        }
    }
    \foreach \i in {3,4,5}{
        \foreach \j in {0,1,6,7}{
            \draw[blue,thick] ({\i*\gap+0.5*\B},{\j*\gap+\B}) -- ({\i*\gap+0.5*\B},{(\j+1)*\gap});
        }
    }
    \end{tikzpicture}
    \caption{\label{fig:Hierachie}Hierarchical structure: Level 1 and Level 2 connections.}
    \end{figure}
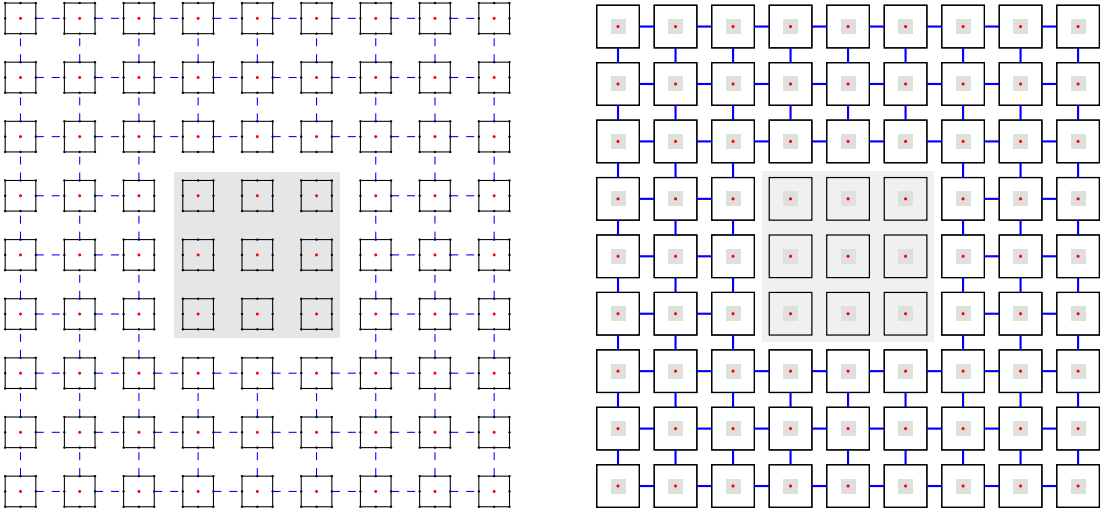

For a level-$0$ block, its distinguished obstruction is the separation
between its central vertex and its exterior cluster; it is crossed when
these belong to the same cluster.

Recursively, let $\mathcal A_k(Q)$ be the central array of level-$(k-1)$
subblocks in a level-$k$ block $Q$. The level-$k$ obstruction is the
collection of the distinguished obstructions of the blocks in
$\mathcal A_k(Q)$. A subblock is called neutralized once its distinguished
obstruction is crossed through an interaction with the level-$k$ cluster.
The level-$k$ obstruction is crossed once the level-$k$ cluster connects
the designated inner and outer parts of $Q$ through $\mathcal A_k(Q)$.
We use two deterministic properties of this construction. First, in order to
cross the level \(k\) obstruction, the dynamics must neutralize a positive
fraction of the obstructing level \(k-1\) blocks in the central array. Second,
one jump of the level \(k\) cluster can neutralize only a bounded number of
such blocks. Thus there is a constant \(c_*>0\), depending only on \(d\), such
that if \(Z\) obstructing subblocks remain closed up to time \(k\), then crossing
the level \(k\) obstruction before time \(k\) forces at least \(c_*Z\) jumps of
the level \(k\) cluster.

Fix \(k\), and suppose the construction below level \(k\) has already been
executed. Let \(\varepsilon_n>0\) be the probability that a single level
\(k-1\) block, run with its internal clocks only, has not crossed its own
distinguished obstruction by time \(k\). The important observation is that
\(\varepsilon_n\) depends on the previous scales and on \(k\), but not on the
new parameter \(N=N_k\). Positivity of $\varepsilon_n$ is immediate, for example from the event
that none of the clocks relevant to this block rings before time \(k\).

Let \(Z_N\) be the number of central level \(k-1\) subblocks which would remain
closed up to time \(k\) under their internal evolution. These events are
independent for the \(N^d\) copies, so
\[
    Z_N\sim \operatorname{Bin}(N^d,\varepsilon_n).
\]
By Chernoff's bound,
\[
    \mathbb P\left(Z_N<\frac12\varepsilon_n N^d\right)
    \le \exp(-c\varepsilon_n N^d).
\]
Choosing \(N\) large, this probability is at most \(1/(2k^2)\).

On the complementary event, crossing the level \(k\) obstruction before time
\(k\) requires at least
\[
    M_N:=c_*\varepsilon_n N^d
\]
jumps of the level \(k\) cluster, after decreasing \(c_*\) if necessary. Before
the obstruction is crossed, the size of this cluster is between two constants
times \(N^d\), where the constants may depend on the already fixed scale \(n\),
but not on \(N\). Therefore its jump rate is bounded from above by
\[
    \lambda_N\le c_n (N^d)^{-\alpha}
\]
for some \(c_n<\infty\) independent of \(N\). Hence the number of such jumps by
time \(k\) is stochastically dominated by
\[
    \operatorname{Pois}\bigl(kc_n(N^d)^{-\alpha}\bigr).
\]
Since
\[
    \frac{kc_n(N^d)^{-\alpha}}{M_N}
    \le
    \frac{kc_n}{c_*\varepsilon_n}N^{-d(1+\alpha)}
    \longrightarrow 0
    \qquad\text{as }N\to\infty,
\]
using \(\alpha>-1\), the Poisson large-deviation bound gives, for \(N\) large
enough,
\[
    \mathbb P\left(
        \operatorname{Pois}\bigl(kc_n(N^d)^{-\alpha}\bigr)\ge M_N
    \right)
    \le \frac1{2k^2}.
\]
Thus \(N=N_k\) can be chosen so that
\[
    \mathbb P\left(
        \text{the level \(k\) obstruction is crossed before time } k
    \right)
    \le \frac1{k^2}.
\]

Choosing the sequence \(N_k\) recursively in this way, we have
\[
    \sum_{k=1}^\infty
    \mathbb P\left(
        \text{the level \(k\) obstruction is crossed before time } k
    \right)
    <\infty.
\]
By Borel--Cantelli, almost surely only finitely many of these events occur. Fix
an integer \(T\). For all sufficiently large \(k\ge T\), the level \(k\)
obstruction has not been crossed by time \(k\), and hence has not been crossed
by time \(T\). Since any infinite cluster must cross every sufficiently large
level obstruction, there is no infinite cluster at time \(T\). 
Therefore no infinite cluster appears at any finite time.
\end{proof}
\section{Fully packed one dimensional Cluster-Cluster model}\label{sec:fullypacked}

 Throughout this section we assume that \(d=1\) and \(p=1\), and no edges at time $0$. Before explosion,
the clusters form a partition of \(\mathbb Z\) into finite intervals. If two
adjacent clusters have sizes \(i\) and \(j\), then their common boundary edge
opens at rate
\[
    q_\alpha(i,j):=\frac12\bigl(i^{-\alpha}+j^{-\alpha}\bigr),
\]
because either cluster may ring and choose the direction of the other. We first
prove that, at each fixed time before explosion, the sequence of cluster
lengths is a stationary renewal process. This identifies the model with a
deterministic time change of Smoluchowski's coagulation equation with the
generalized sum kernel
\begin{equation}\label{def:SmulKernel}
     K_\alpha(i,j):=i^{-\alpha}+j^{-\alpha}.
\end{equation}
We then use classical results for this equation whenever they are available,
and prove the estimates that are specific to the physical time scale of
the Cluster-Cluster model.

\subsection{Renewal structure and the coagulation equation}

For every edge \(e=\{x,x+1\}\), let
\[
    D_t(e):=\{e\text{ is closed at time }t\},
    \qquad
    d(t):=\mathbb P(D_t(\{0,1\}))  = \mathbb P(D_t(e)). 
\]
As long as there is no infinite cluster, \(d(t)\) is the cluster density. 

\begin{lemma}
\label{lem:closed-edge-regeneration}
Fix a time \(t\) before explosion and thus \(d(t)>0\). Conditional on
\(D_t(\{0,1\})\), the restrictions of the configuration at time \(t\) to
\(\{\ldots,-1,0\}\) and to \(\{1,2,\ldots\}\) are independent.
\end{lemma}

\begin{proof}
Split the clock of each cluster into independent left and right pointing
clocks, each with half of the original rate. Suppress the edge \(\{0,1\}\),
and construct the two half-line dynamics independently. Let \(L_s\) and
\(R_s\) be the sizes of the clusters adjacent to the suppressed edge at time
$s\in [0,t]$. Independently of the two half-line evolutions, let \(N_-\) and \(N_+\)
be rate $1$ Poisson processes. The attempted crossings of \(\{0,1\}\) from
the left and right are obtained by the time changes
\[
    N_-\left(\frac12\int_0^s L_u^{-\alpha}\,du\right)
    \quad\text{and}\quad
    N_+\left(\frac12\int_0^s R_u^{-\alpha}\,du\right),
\]
respectively. Hence \(D_t(\{0,1\})\) is the intersection of an event measurable
with respect to the left-half-line evolution and \(N_-\), and an event
measurable with respect to the right-half-line evolution and \(N_+\) up to time $t$.
Conditioning on the independent events 
\[
   \left\{ N_-\left(\frac12\int_0^t L_u^{-\alpha}\,du\right) = 0 \right\}
\]
and
\[
   \left\{ N_+\left(\frac12\int_0^t R_u^{-\alpha}\,du\right) = 0 \right\}
\]
gives us the distribution conditioned on $D_t(\{0,1\})$ and
preserves independence between the two half-line configurations at time $t$.

Thus, conditioned on $D_t(0,1)$ the two half-lines are independent.
\end{proof}

Let \(V_k(t)\) denote the particle density in clusters of size \(k\), i.e. by ergodicity, 
$V_k(t) = \mathbb{P}\big( |\mathfs{C}_0(t)| = k \big),$
and let
\[
    c_k(t):=\frac{V_k(t)}{k}
\]
be the density of clusters of size \(k\), where we recall that  $\mathfs{C}_0(t)$ denotes the cluster started at $0$. Thus
\[
    d(t)=\sum_{k\ge1}c_k(t).
\]

\begin{corollary}
\label{cor:renewal-law}
Fix \(t\) before explosion, and thus \(d(t)>0\).  Conditional on
\(D_t(\{0,1\})\), the successive cluster lengths to the right of
\(\{0,1\}\) are independent and identically distributed with law
\[
    p_k(t):=\frac{c_k(t)}{d(t)},
    \qquad k\ge1.
\]
The same is true to the left, independently. Consequently, if
\(W_{i,j}(t)\) is the density of closed edges whose left and right clusters
have sizes \(i\) and \(j\), respectively, then
\begin{equation}
\label{eq:pair-density}
    W_{i,j}(t)
    =d(t)p_i(t)p_j(t)
    =\frac{c_i(t)c_j(t)}{d(t)}.
\end{equation}
\end{corollary}

\begin{proof}
The cluster immediately to the right of a closed edge has length \(k\) with
conditional probability \(c_k(t)/d(t)\), because every cluster has exactly one
left boundary. After the next closed edge, Lemma~\ref{lem:closed-edge-regeneration}
restarts the same conditional law independently. Induction gives the renewal
law, and applying the regeneration lemma on both sides of \(\{0,1\}\) gives
\eqref{eq:pair-density}.
\end{proof}

Let us now use these insights to study the behaviour of the Cluster-Cluster model for different values of $\alpha$.
For \(\alpha\ge-1\), all rates appearing below are integrable on compact
pre-explosion intervals. Indeed,
\[
    \sum_{k\ge1}k^{-\alpha}c_k(t) \leq \sum_{k\ge1}kc_k(t) = \frac{1}{d(t)} <\infty,
\]
since \(k^{-\alpha}\le 1\) for \(\alpha\ge0\), while
\(k^{-\alpha}\le k\) for \(-1\le\alpha<0\). The following (deterministic) balance equations
therefore hold in integrated form, and hence almost everywhere in time. Recall $K_{\alpha}$ from \eqref{def:SmulKernel}. The following proposition gives the physical time coagulation equation. 

\begin{proposition}
\label{prop:physical-coagulation-equation}
For \(\alpha\ge-1\) and every \(k\ge1\),
\begin{equation}
\label{eq:physical-coagulation}
\begin{aligned}
    c_k'(t)
    =\frac{1}{2d(t)}
      \sum_{i=1}^{k-1}K_\alpha(i,k-i)c_i(t)c_{k-i}(t) -\frac{c_k(t)}{d(t)}
      \sum_{j\ge1}K_\alpha(k,j)c_j(t),
\end{aligned}
\end{equation}
with \(c_1(0)=1\) and \(c_k(0)=0\) for \(k\ge2\). Moreover,
\begin{equation}
\label{eq:density-derivative}
    d'(t)=-\sum_{k\ge1}k^{-\alpha}c_k(t).
\end{equation}
Equivalently, the particle densities satisfy
\begin{equation}
\label{eq:physical-particle-density}
\begin{aligned}
    V_k'(t)
    &=\frac{k}{2d(t)}\sum_{i=1}^{k-1}
      \frac{K_\alpha(i,k-i)}{i(k-i)}V_i(t)V_{k-i}(t) 
      -\frac{kV_k(t)}{d(t)}\sum_{j\ge1}  \frac{K_\alpha(k,j)}{k j} V_j(t).
\end{aligned}
\end{equation}

\end{proposition}

\begin{proof}
By \eqref{eq:pair-density}, boundaries with adjacent cluster sizes \(i\) and
\(j\) have density \(c_i c_j/d\), and each such boundary opens at rate
\(K_\alpha(i,j)/2\) and adds $k$ vertices in clusters of size $k$ (thus multiplication by $k$ in \eqref{eq:physical-particle-density}). The loss part is composed from all boundaries with adjacent sizes $k$ and any $i$. Summing the corresponding gains and losses gives
\eqref{eq:physical-coagulation}. Summing that equation over \(k\) yields
\eqref{eq:density-derivative}; using
\(c_j=V_j/j\) gives \eqref{eq:physical-particle-density}.
\end{proof}

Next, consider the time normalized by the cluster density $p(t)$ and defined by:
\begin{equation}\label{eq:clusterdensity definition}
    p_k(t):=\frac{c_k(t)}{d(t)},
    \qquad
    \mu(t):=\sum_{k\ge1}k p_k(t),
\end{equation}
i.e. $\mu(t)=\BE\bb{|\mathfs{C}_0(t)|}$ is the expectation of the non-size biased distribution of the clusters size. 
Note that before explosion, mass conservation gives
\begin{equation}
\label{eq:density-mean-relation}
    d(t)\mu(t)=1.
\end{equation}

\begin{proposition}
\label{prop:triangular-equation}
For every \(k\ge1\),
\begin{equation}
\label{eq:triangular-equation}
    p_k'(t)
    =\sum_{i=1}^{k-1}i^{-\alpha}p_i(t)p_{k-i}(t)
     -k^{-\alpha}p_k(t),
    \qquad
    p_k(0)=\mathbf 1_{\{k=1\}}.
\end{equation}
In particular,
\begin{equation}
\label{eq:triangular-recursion}
\begin{aligned}
    p_1(t)&=e^{-t},\\
    p_k(t)&=\int_0^t e^{-k^{-\alpha}(t-s)}
      \sum_{i=1}^{k-1}i^{-\alpha}p_i(s)p_{k-i}(s)\,ds,
      \qquad k\ge2.
\end{aligned}
\end{equation}
\end{proposition}

\begin{proof}
Differentiate \(p_k=c_k/d\), use
\eqref{eq:physical-coagulation}--\eqref{eq:density-derivative}, and note that
\[
    \frac12\sum_{i=1}^{k-1}
    \bigl(i^{-\alpha}+(k-i)^{-\alpha}\bigr)p_i p_{k-i}
    =\sum_{i=1}^{k-1}i^{-\alpha}p_i p_{k-i}.
\]
The terms involving the mean rate of a neighboring cluster cancel.
Variation of constants gives \eqref{eq:triangular-recursion}.
\end{proof}

Recall the gelation time from \eqref{def:Gelation} defined by 
\[
    T_\alpha=\sup\{t\ge0:d(t)>0\}.
\]
At a fixed time, if \(d(t)>0\), the ergodic theorem gives infinitely many
closed edges in both directions, and hence every cluster is finite. If
\(d(t)=0\), then every fixed edge is open almost surely, and therefore all
edges are open almost surely. Thus \(T_\alpha\) is exactly the explosion time.

For later use, let \(X_t\) have law \(p(t)\) (see \eqref{eq:clusterdensity definition}), and define
\begin{equation}\label{eq:MomentBound}
     A_r(t):=\mathbb E[X_t^{-r}],\qquad r>0.
\end{equation}
Whenever \(\mu(t)<\infty\), summation of
\eqref{eq:triangular-equation} against \(k\) gives

\[
\begin{aligned}
    \mu'(t)
    &=
    \sum_{k\ge1}k p_k'(t) \\
    &=
    \sum_{i,j\ge1}
    (i+j)i^{-\alpha}p_i(t)p_j(t)
    -
    \sum_{i\ge1}i^{1-\alpha}p_i(t) \\
    &=
    \sum_{i,j\ge1}
    i^{1-\alpha}p_i(t)p_j(t)
    +
    \sum_{i,j\ge1}
    j i^{-\alpha}p_i(t)p_j(t)
    -
    \sum_{i\ge1}i^{1-\alpha}p_i(t).
\end{aligned}
\]
Since \(\sum_{j\ge1}p_j(t)=1\), the first and third terms cancel.
Therefore
\begin{equation}
\label{eq:mean-equation-alpha-positive}
    \mu'(t)
    =
    \mu(t)A_\alpha(t).
\end{equation}


Next, define the Smoluchowski time
\begin{equation}
\label{eq:smol-time-change}
    \tau(t):=\int_0^t\frac{ds}{d(s)}.
\end{equation}
Let \(t(\tau)\) be its inverse and set
\(\widehat c_k(\tau):=c_k(t(\tau))\). Then, by the chain rule,
\begin{equation}
\label{eq:standard-smoluchowski}
\begin{aligned}
    \frac{d}{d\tau}\widehat c_k(\tau)
    &=\frac12\sum_{i=1}^{k-1}
      K_\alpha(i,k-i)\widehat c_i(\tau)\widehat c_{k-i}(\tau) \\
    &\quad-\widehat c_k(\tau)
      \sum_{j\ge1}K_\alpha(k,j)\widehat c_j(\tau),
\end{aligned}
\end{equation}
and physical time is recovered from
\begin{equation}
\label{eq:inverse-time-change}
    t(\tau)=\int_0^\tau\widehat d(s)\,ds,
    \qquad
    \widehat d(\tau):=\sum_{k\ge1}\widehat c_k(\tau).
\end{equation}
For \(\alpha\ge-1\), the kernel satisfies
\(K_\alpha(i,j)\le C(i+j)\). Thus global mass-conserving solutions of
\eqref{eq:standard-smoluchowski} follow from
\cite[Theorem~2.3]{Laurencot2002}; this is the coagulation-only specialization
of the original result \cite[Theorems~2.4 and~2.5]{BallCarr1990}.

\subsection{The phase diagram}

\subsubsection{\bf{The regime} \(\alpha>0\)}

The following sharp physical-time estimate is specific to the one-dimensional Cluster-Cluster model;
the exponent agrees with the homogeneity heuristic for the generalized sum
kernel, but the proof uses 
\eqref{eq:triangular-equation}.

\begin{proposition}
\label{thm:positive-alpha-uniform-cluster}
Let \(\alpha>0\). Then \(T_\alpha=\infty\), and there are constants
\(0<c_\alpha\le C_\alpha<\infty\) such that
\[
    c_\alpha(1+t)^{1/\alpha}
    \le \mu(t)\le
    C_\alpha(1+t)^{1/\alpha},
    \qquad t\ge0.
\]
More precisely,
\[
    \mu(t)\ge(1+\alpha t)^{1/\alpha}.
\]
\end{proposition}
\begin{remark}
In particular for any given $\alpha>0$, we have that \(d(t)\asymp(1+t)^{-1/\alpha}\). This is sufficient for the results presented in this paper. However, in order to understand asymptotics as $\alpha\to0$ one needs to use finer estimates. 
\end{remark}

\begin{proof}[Proof of Proposition~\ref{thm:positive-alpha-uniform-cluster}]
Since \(\alpha>0\), the function \(x\mapsto x^{-\alpha}\) is convex on
\((0,\infty)\). Jensen's inequality gives
\[
    A_\alpha(t)
    =
    \mathbb{E}[X_t^{-\alpha}]
    \ge
    \mathbb{E}[X_t]^{-\alpha}
    =
    \mu(t)^{-\alpha}.
\]
Combining this with \eqref{eq:mean-equation-alpha-positive}, we obtain
\[
    \mu'(t)
    \ge
    \mu(t)^{1-\alpha}.
\]
Consequently,
\[
    \frac{d}{dt}\mu(t)^\alpha
    =
    \alpha\mu(t)^{\alpha-1}\mu'(t)
    \ge
    \alpha.
\]
Since the initial configuration consists entirely of singleton clusters,
\(\mu(0)=1\). Integrating from \(0\) to \(t\) therefore yields
\[
    \mu(t)^\alpha
    \ge
    1+\alpha t.
\]
Hence
\begin{equation}
\label{eq:positive-alpha-lower}
    \mu(t)
    \ge
    \left(1+\alpha t\right)^{1/\alpha}.
\end{equation}

For completeness, note that \(X_t\ge1\), and hence
\(A_\alpha(t)\le1\). Thus
\[
    \mu'(t)\le\mu(t),
\]
so that \(\mu(t)\le e^t\) on every finite time interval. In particular,
\(\mu(t)\) cannot diverge in finite time, and the preceding calculation is
valid for every \(t\ge0\).

For fixed \(r>0\), differentiating \eqref{eq:MomentBound} with respect to $t$ and using
\eqref{eq:triangular-equation}, gives for independent
\(X_t,Y_t\sim p(t)\),
\begin{equation}
\label{eq:negative-moment-identity}
    A_r'(t)=\mathbb E\left[
      X_t^{-\alpha}\bigl((X_t+Y_t)^{-r}-X_t^{-r}\bigr)
    \right].
\end{equation}
On \(\{Y_t\ge X_t\}\), the expression $\bigl(X_t^{-r}-(X_t+Y_t)^{-r}\bigr)$ is at least
\((1-2^{-r})X_t^{-r}\). Since
\(x\mapsto x^{-\alpha-r}\) and
\(x\mapsto\mathbb P(Y_t\ge x)\) are both non-increasing, their covariance is
nonnegative. Also \(\mathbb P(Y_t\ge X_t)\ge1/2\) by symmetry. Hence
\[
    -A_r'(t)\ge \frac{1-2^{-r}}2
      \mathbb E[X_t^{-\alpha-r}]
    \ge c_r A_r(t)^{1+\alpha/r},
    \qquad c_r:=\frac{1-2^{-r}}2,
\]
where the last inequality is Jensen's. Therefore
\begin{equation}
\label{eq:negative-moment-decay}
    A_r(t)\le
    \left(1+\frac{\alpha c_r}{r}t\right)^{-r/\alpha}.
\end{equation}


For the reverse inequality, define, for \(0\le z<1\),
\[
    P(z,t):=\sum_{k\ge1}p_k(t)z^k,
    \qquad
    B(z,t):=\sum_{k\ge1}k^{-\alpha}p_k(t)z^k.
\]
Equation \eqref{eq:triangular-equation} implies
\[
    \partial_tP(z,t)=B(z,t)(P(z,t)-1),
\]
and therefore
\begin{equation}
\label{eq:generating-function-integral}
    \int_0^t B(z,s)\,ds
    =\log\frac{1-P(z,t)}{1-z}
    \le\log\frac1{1-z}.
\end{equation}
On the other hand, \(\log\mu(t)=\int_0^tA_\alpha(s)\,ds\).

If \(0<\alpha\le1\), then
\(1-z^k\le(k(1-z))^\alpha\), and hence
\(A_\alpha(s)-B(z,s)\le(1-z)^\alpha\). Thus
\[
    \log\mu(t)\le \log\frac1{1-z}+t(1-z)^\alpha.
\]
Taking \(1-z=(1+t)^{-1/\alpha}\) gives
\(\mu(t)\le e(1+t)^{1/\alpha}\).

If \(\alpha>1\), then
\[
    A_\alpha(s)-B(z,s)
    \le(1-z)A_{\alpha-1}(s)
    \le C_\alpha(1-z)(1+s)^{-(\alpha-1)/\alpha}
\]
by \eqref{eq:negative-moment-decay}. Hence
\[
    \log\mu(t)
    \le\log\frac1{1-z}
      +C_\alpha(1-z)(1+t)^{1/\alpha}.
\]
The same choice \(1-z=(1+t)^{-1/\alpha}\) proves the upper bound. In
particular, \(\mu\) cannot diverge at finite time, so \(T_\alpha=\infty\).
\end{proof}

\subsubsection{\bf{The constant kernel} \(\alpha=0\)} \label{subsec:alpha=0}

For \(\alpha=0\), \eqref{eq:standard-smoluchowski} has constant kernel
\(K_0=2\). The monodisperse solution is classical; see
\cite{DeaconuTanre2000}. Applying the inverse time change
\eqref{eq:inverse-time-change} gives
\begin{equation}
\label{eq:alpha-zero-exact}
    d(t)=e^{-t},
    \qquad
    p_k(t)=e^{-t}(1-e^{-t})^{k-1},
    \qquad
    \mu(t)=e^t.
\end{equation}
Thus \(T_0=\infty\). Moreover, the cluster containing the origin is the sum
of the independent open runs to its left and right, and hence
\[
    \mu(t)=\mathbb E[|\mathfs C_0(t)|]=2e^t-1.
\]
\subsubsection{\bf{The regime \(-1<\alpha<0\)}}
\label{subsubsec:sublinear-accelerating-regime}

Throughout this subsection, set
\[
    q:=-\alpha\in(0,1).
\]
The associated Smoluchowski kernel is
\[
    K_q(i,j)=i^q+j^q.
\]
Since \(i^q\le i\) for \(i\ge1\), this kernel has at most linear growth.
Consequently, the classical theory for the discrete Smoluchowski equation
provides a unique global mass-conserving solution in Smoluchowski time; see
\cite{BallCarr1990,Norris1999}.

Let \((\widehat c_k(\tau))_{k\ge1}\) denote this solution with monodisperse
initial condition
\[
    \widehat c_1(0)=1,
    \qquad
    \widehat c_k(0)=0,\quad k\ge2,
\]
and write
\[
    \widehat d(\tau):=\sum_{k\ge1}\widehat c_k(\tau).
\]
Recall that physical time is given by
\[
    t(\tau):=\int_0^\tau \widehat d(s)\,ds.
\]
Note that the gelation time satisfies
\[
    T_q=\lim_{\tau\to\infty}t(\tau)
    =
    \int_0^\infty\widehat d(\tau)\,d\tau.
\]
For \(t<T_q\), recall that \(\tau(t)\) is the inverse of \(t(\tau)\), and that
\[
    c_k(t)=\widehat c_k(\tau(t)),
    \qquad
    d(t)=\sum_{k\ge1}c_k(t),
    \qquad
    p_k(t)=\frac{c_k(t)}{d(t)}.
\]
By mass conservation,
\[
    \sum_{k\ge1}k c_k(t)=1,
\]
and hence
\begin{equation}
\label{eq:sublinear-density-mean}
    \mu(t)=\sum_{k\ge1}k p_k(t)=\frac1{d(t)}.
\end{equation}
The probability distribution $p(t)$ from \eqref{eq:clusterdensity definition} satisfies the triangular equation \eqref{eq:triangular-equation}.

The main result of this subsection is the following.

\begin{proposition}
\label{thm:sublinear-accelerating}
Let \(q=-\alpha\in(0,1)\). Then
\[
    0<T_q<\infty.
\]
More precisely, define
\begin{equation}
\label{eq:kappa-q-definition}
    \kappa_q
    :=
    \inf_{0<u<1}
    \frac{
        u^q+(1-u)^q-u^{2q}-(1-u)^{2q}
    }{
        2u^q(1-u)^q
    }.
\end{equation}
Then \(\kappa_q>0\), and
\begin{equation}
\label{eq:gel-time-bounds-sublinear}
    \frac1q
    \le
    T_q
    \le
    \min\left\{
        \frac1{\kappa_q},
        \frac1{1-2^{-q}}
    \right\}.
\end{equation}

Furthermore, if
\[
    M_q(t):=A_{\alpha}(t)=\sum_{k\ge1}k^q p_k(t),
\]
then, for every \(t<T_q\),
\begin{equation}
\label{eq:Mq-two-sided-near-explosion}
    \frac1{T_q-t}
    \le
    M_q(t)
    \le
    \frac1{\kappa_q(T_q-t)}.
\end{equation}
Consequently,
\begin{equation}
\label{eq:mu-two-sided-sublinear}
    \bigl(q(T_q-t)\bigr)^{-1/q}
    \le
    \mu(t)
    \le
    \left(
        \frac{T_q}{T_q-t}
    \right)^{1/\kappa_q},
    \qquad 0\le t<T_q,
\end{equation}
and
\begin{equation}
\label{eq:d-two-sided-sublinear}
    \left(
        \frac{T_q-t}{T_q}
    \right)^{1/\kappa_q}
    \le
    d(t)
    \le
    \bigl(q(T_q-t)\bigr)^{1/q}.
\end{equation}
\end{proposition}

\begin{proof}[Proof of Proposition~\ref{thm:sublinear-accelerating}]
We divide the proof into four steps.

\medskip

\noindent
\textbf{Step 1: a lower bound on \(T_q\).}

By concavity of \(x\mapsto x^q\),
\(
    M_q(t)=\mathbb E[X_t^q]\le \mu(t)^q.
\)
Together with \eqref{eq:mean-equation-alpha-positive} 
 this gives
\(
    \mu'(t)\le \mu(t)^{q+1}.
\)
Since \(\mu(0)=1\),
\begin{equation}
\label{eq:mu-pre-explosion-upper}
    \mu(t)
    \le
    (1-qt)^{-1/q},
    \qquad
    0\le t<\min\{T_q,1/q\}.
\end{equation}

We claim that \(T_q\ge1/q\). To see this, suppose that \(T_q<1/q\), and note that in this case 
\eqref{eq:mu-pre-explosion-upper} would imply
\(
    \sup_{t<T_q}\mu(t)<\infty.
\)
Hence
\[
    \inf_{t<T_q}d(t)
    =
    \inf_{t<T_q}\frac1{\mu(t)}
    >0.
\]
Along the global Smoluchowski solution,
\[
    \frac{dt}{d\tau}=\widehat d(\tau)
    =
    d(t(\tau)).
\]
Therefore \(dt/d\tau\) would be bounded below by a positive constant, forcing
\(t(\tau)\to\infty\) as \(\tau\to\infty\). This contradicts
\(t(\tau)\uparrow T_q<\infty\), and we obtain the claim.

\medskip

\noindent
\textbf{Step 2: two upper bounds on \(T_q\).}

First define the negative moment
\[
    A_q(t):=\mathbb E[X_t^{-q}].
\]
Then \eqref{eq:negative-moment-identity} gives
\[
\begin{aligned}
    A_q'(t)
    =
    \mathbb E\left[
        X_t^q
        \bigl((X_t+Y_t)^{-q}-X_t^{-q}\bigr)
    \right] =
    \mathbb E\left[
        \left(
            \frac{X_t}{X_t+Y_t}
        \right)^q
    \right]-1.
\end{aligned}
\]
By symmetry,
\[
\begin{aligned}
    \mathbb E\left[
        \left(
            \frac{X_t}{X_t+Y_t}
        \right)^q
    \right]
    &=
    \frac12
    \mathbb E\left[
        \left(
            \frac{X_t}{X_t+Y_t}
        \right)^q
        +
        \left(
            \frac{Y_t}{X_t+Y_t}
        \right)^q
    \right].
\end{aligned}
\]
For \(u\in[0,1]\), concavity of \(x\mapsto x^q\) gives
\[
    \frac{u^q+(1-u)^q}{2}\le 2^{-q}.
\]
Consequently,
\[
    A_q'(t)\le -(1-2^{-q}).
\]
Since \(A_q(0)=1\) and \(A_q(t)\ge0\), the solution cannot persist beyond
time \((1-2^{-q})^{-1}\). Hence
\begin{equation}
\label{eq:Tq-upper-negative-moment}
    T_q\le\frac1{1-2^{-q}}.
\end{equation}

We next obtain the second upper bound from \(M_q\). By subadditivity,
\[
    (x+y)^q-x^q\le y^q,
\]
and therefore \eqref{eq:negative-moment-identity} gives
\begin{equation}
\label{eq:Mq-upper-Riccati}
    M_q'(t)\le M_q(t)^2.
\end{equation}

For the reverse inequality, symmetrize:
\[
\begin{aligned}
    2M_q'(t)
    &=
    \mathbb E\left[
        (X_t^q+Y_t^q)(X_t+Y_t)^q
        -X_t^{2q}-Y_t^{2q}
    \right].
\end{aligned}
\]
For \(x,y>0\), put \(u=x/(x+y)\). Then
\[
\begin{aligned}
    &(x^q+y^q)(x+y)^q-x^{2q}-y^{2q} \\
    &\qquad=
    (x+y)^{2q}
    \bigl(
        u^q+(1-u)^q-u^{2q}-(1-u)^{2q}
    \bigr).
\end{aligned}
\]
By the definition of \(\kappa_q\),
\[
    (x^q+y^q)(x+y)^q-x^{2q}-y^{2q}
    \ge
    2\kappa_q x^qy^q.
\]
Thus
\begin{equation}
\label{eq:Mq-lower-Riccati}
    M_q'(t)\ge \kappa_q M_q(t)^2.
\end{equation}

It remains to verify that \(\kappa_q>0\). Indeed, one can directly verify that the function
\[
    h_q(u)
    :=
    \frac{
        u^q+(1-u)^q-u^{2q}-(1-u)^{2q}
    }{
        2u^q(1-u)^q
    }
\]
is continuous and strictly positive on \((0,1)\).
Moreover, recalling that  $q<1$, 
\[
    \lim_{u\downarrow0}h_q(u)
    =
    \lim_{u\uparrow1}h_q(u)
    =
    \frac12.
\]
Hence \(h_q\) extends to a strictly positive continuous
function on \([0,1]\), and therefore \(\kappa_q>0\).

Since \(M_q(0)=1\), \eqref{eq:Mq-lower-Riccati} implies
\[
    \frac1{M_q(t)}
    \le 1-\kappa_qt
\]
as long as the solution exists. Thus
\begin{equation}
\label{eq:Tq-upper-Mq}
    T_q\le\frac1{\kappa_q}.
\end{equation}
Together, \eqref{eq:Tq-upper-negative-moment} and
\eqref{eq:Tq-upper-Mq} prove the upper bound in
\eqref{eq:gel-time-bounds-sublinear}.

\medskip

\noindent
\textbf{Step 3: behavior of \(M_q\) near \(T_q\).}

We first claim that
\begin{equation}
\label{eq:Mq-diverges}
    M_q(t)\longrightarrow\infty
    \qquad\text{as }t\uparrow T_q.
\end{equation}
Indeed, \(M_q\) is nondecreasing by
\eqref{eq:negative-moment-identity}. If it remained bounded on
\([0,T_q)\), then \eqref{eq:mean-equation-alpha-positive} and Gronwall's inequality
would imply that \(\mu\) remains bounded on \([0,T_q)\). As in Step 1, this
would force \(t(\tau)\to\infty\), contradicting \(T_q<\infty\).

From \eqref{eq:Mq-upper-Riccati},
\[
    \frac{d}{dt}\frac1{M_q(t)}
    =
    -\frac{M_q'(t)}{M_q(t)^2}
    \ge -1.
\]
Integrating from \(t\) to \(s<T_q\) and then letting \(s\uparrow T_q\), using
\eqref{eq:Mq-diverges}, gives
\[
    \frac1{M_q(t)}
    \le T_q-t.
\]
Similarly, \eqref{eq:Mq-lower-Riccati} gives
\[
    \frac{d}{dt}\frac1{M_q(t)}
    \le -\kappa_q.
\]
Integration from \(t\) to \(T_q\), and rearranging the inequality, yields
\[
    M_q(t)\le\frac1{\kappa_q(T_q-t)}.
\]
allowing us to conclude \eqref{eq:Mq-two-sided-near-explosion}.

\medskip

\noindent
\textbf{Step 4: behavior of \(\mu\) and \(d\).}

Since \(\mu(0)=1\), equation \eqref{eq:mean-equation-alpha-positive} gives
\[
    \log\mu(t)=\int_0^t M_q(s)\,ds.
\]
Using the upper bound in
\eqref{eq:Mq-two-sided-near-explosion},
\[
\begin{aligned}
    \log\mu(t)
    \le
    \frac1{\kappa_q}
    \int_0^t\frac{ds}{T_q-s} =
    \frac1{\kappa_q}
    \log\left(\frac{T_q}{T_q-t}\right).
\end{aligned}
\]
giving the desired upper bound on $\mu(t)$. 

For the lower bound, return to
\[
    \mu'(t)\le\mu(t)^{q+1}.
\]
By the argument in Step 3, $\mu(t)\rightarrow\infty$ as $t\uparrow T_q$.
Integrating from \(t\) to \(T_q\) therefore gives $\mu(t)^{-q}\le q(T_q-t),$ or $\mu(t)\ge
    \bigl(q(T_q-t)\bigr)^{-1/q}$. 
This proves \eqref{eq:mu-two-sided-sublinear}. The density bounds
\eqref{eq:d-two-sided-sublinear} follow immediately from
\(d(t)=1/\mu(t)\).

Finally, \(d(t)\downarrow0\) as \(t\uparrow T_q\). For every fixed edge \(e\),
\[
    \mathbb P(e\text{ is closed at time }t)=d(t).
\]
Since edges only open, continuity from above gives
\[
    \mathbb P(e\text{ remains closed up to }T_q)
    =
    \lim_{t\uparrow T_q}d(t)=0.
\]
There are countably many edges in \(\mathbb Z\), so almost surely every edge
is open by time \(T_q\). Thus \(T_q\) is the explosion time of the 
Cluster-Cluster model.
\end{proof}

\begin{remark}
The kernel \(K_q(x,y)=x^q+y^q\), \(q\in(0,1)\), is non-gelling in the usual
Smoluchowski time. Self-similar profiles for this kernel, including their
regularity, exponential decay at infinity, and behavior near the origin, were
studied in \cite{FournierLaurencot2006}. Those results concern properties of
self-similar profiles and do not by themselves imply convergence of the
monodisperse solution to such a profile. Accordingly, we do not claim the
matching asymptotic
\[
    \mu(t)\asymp (T_q-t)^{-1/q}.
\]
The bounds in \eqref{eq:mu-two-sided-sublinear} are the quantitative estimates
proved here directly from the triangular equation.
\end{remark}

\subsubsection{\bf{The additive kernel} \(\alpha=-1\)}\label{subsec:alpha-1}

For \(\alpha=-1\), \eqref{eq:standard-smoluchowski} has the additive kernel
\(K(i,j)=i+j\). The monodisperse solution is classical; see
\cite{DeaconuTanre2000}. Since its cluster density in Smoluchowski time is
\(e^{-\tau}\), the inverse time change gives \(t=1-e^{-\tau}\). Consequently,
for \(0\le t<1\),
\begin{equation}
\label{eq:alpha-minus-one-exact}
    T_{-1}=1,
    \qquad
    d(t)=1-t,
    \qquad
    \mu(t)=\frac1{1-t},
\end{equation}
and the law of a uniformly chosen cluster is
\[
    p_k(t)=\frac{(tk)^{k-1}}{k!}e^{-tk},
    \qquad k\ge1.
\]
The size-biased law of the cluster containing the origin therefore satisfies
\[
    \mathbb E[|\mathfs{C}_0(t)|]=\frac1{(1-t)^2}.
\]
At time \(1\), every edge is open almost surely.

\subsubsection{\bf{The regime} \(\alpha<-1\)}\label{subsec:alpha<-1}

Write \(q=-\alpha>1\). Carr and da Costa \cite{CarrDaCosta1992} prove instantaneous gelation for
discrete coagulation kernels bounded below by a generalized product kernel
with one exponent larger than one; their result applies to
\(K(i,j)=i^q+j^q=i^qj^0+i^0j^q\). If the spatial process were explosion-free on a
nontrivial physical-time interval, the renewal reduction and
\eqref{eq:smol-time-change} would produce a mass-conserving solution of the
corresponding Smoluchowski equation on a nontrivial Smoluchowski-time interval,
contradicting that theorem. Hence for $\alpha<-1$, 
\(
    T_\alpha=0.
\)
\begin{proof}[Proof of Theorem \ref{thm:one-dimensional-phase-diagram}]
(1) is covered by Proposition~\ref{thm:positive-alpha-uniform-cluster}. (2) is covered by Section~\ref{subsec:alpha=0}. (3) is covered by Proposition~\ref{thm:sublinear-accelerating}
 (4) is covered by Section~\ref{subsec:alpha-1}. (5) is covered by Section~\ref{subsec:alpha<-1}.
\end{proof}

\addtocontents{toc}{\protect\setcounter{tocdepth}{0}}

\section*{Acknowledgments}
The authors would like to thank Gidi Amir for introducing us to the model and questions and to Balazs Rath for very useful comments during the "Workshop on Disordered media" that took place at the Renyi Institute. This project was partially supported by DFG grant 5010383 and COST grant CA24122. Much of this work was done while the third author was funded by the Packard Foundation via Amol Aggarwal’s Packard Fellowships for Science and Engineering and by the Simons Foundation via
Ivan Corwin’s Investigator Award. 

\addtocontents{toc}{\protect\setcounter{tocdepth}{2}}

\bibliography{ri}
\bibliographystyle{plain}

%
%


\appendix
\section{Alternative proof for absence of blow-up}\label{appendix}
 We establish that there is no blowup in the case $\alpha\ge 1$. While only treating a sub-regime of  Theorem~\ref{thm:finite_clusters}, this route has the advantage of yielding effective bounds on the underlying cluster sizes. 
\begin{lemma}\label{lem:infinite doubling}
    Let $\alpha \ge 1$. With probability one, a cluster that has connections solely to finite clusters is finite.
\end{lemma}


The general idea for the proof is to consider the growth rate of $\mathfs{C}_0$ up until doubling. 
Starting from size $m$, we can bound the boundary of the cluster by $4dm$ up until doubling. 
If we look at the expected size increase from each potential cluster on the boundary with size $k$, since $\alpha \ge 1$ we get $k^{1-\alpha} \le 1$.
We also must consider the expected size increase from $\mathfs{C}_0$ moving into another cluster, but since we only care about doubling time, any size above $m$ is identical, hence we can bound by $m^{1-\alpha} \le 1$.
The result is that after $t$ time, we gain $4dmt$ size, which gives that the expectation of the doubling times is bounded below by $\frac{1}{4d}$.  
The technical part is to take the first moment heuristic into high-probability estimation.

\begin{proof}
Define the sequence of {\it doubling times}: 
\begin{equation}
    \begin{aligned}
        \mathcal{T}_1 = 0 \quad \text{and} \quad \mathcal{T}_n = \inf\{t:|\mathfs{C}_0(t)|\ge2|\mathfs{C}_{0}(\mathcal{T}_{n-1})|\},
    \end{aligned}
\end{equation}
with $\mathcal{T}_n=\infty$, if $|\mathfs{C}_{0}(\mathcal{T}_{n-1})|=\infty$.
Denote $\Delta \mathcal{T}_n = \mathcal{T}_{n+1} - \mathcal{T}_n$.
We want to show that there exists  $\delta,\ep >0$, such that for all $n$, $\Prob{\Delta \mathcal{T}_n > \delta|\mathfs{F}_{\tau_n} }\ge \ep$.
To do this, we consider the construction of the process via rejection sampling. 
In this construction, each cluster $\mathfs{C}$ has a rate $1$ Poisson clock  and a generator of $\operatorname{Unif}[0,1]$-distributed random variables. When a clock rings at time $t$, the cluster will move if for the next $U \sim \operatorname{Unif}[0,1]$ random variable when $U < \abs{\mathfs{C}(t)}^{-\alpha}$.
We assert that all clocks and the uniform random variables are independent.
Now consider starting from time $\mathcal{T}_n$. If $\abs{\mathfs{C}_0(\mathcal{T}_n)} = \infty$ then we only have $n<\infty$ doublings, thus there were finitely many connections, and so since $\abs{\mathfs{C}_0(\mathcal{T}_n)} = \infty$, one of the connections must be to an infinite cluster. So, assume $\abs{\mathfs{C}_0(\mathcal{T}_n)} = m$. For simplicity, we will shift the time so that $\mathcal{T}_n = 0$, and denote $\tau = \mathcal{T}_{n+1}$.
Using this construction, $\mathfs{C}_0$ can only grow when one of the two options happens:
\begin{enumerate}
    \item a clock of a direct neighbor rings, at which point, given the size of the neighbor $m'$, it will grow by $m'$ if $U< m'^{-\alpha}$.
    \item $\mathfs{C}_0$'s clock rings, in which case we must have that $U < m^{-\alpha}$. This at worst may lead to doubling, so we can upper bound the size increase by $m$, since any larger increase is equivalent in terms of the next doubling time. Note that this is identical in terms of size to having a clock of a neighbor of size $m$ ring.
\end{enumerate}
Since at any given time up to the doubling the size is bounded by $2m-1$ and thus the outer boundary by $4dm-2d$ there are at most $4dm-2d+1 \le 4dm$ rate $1$ Poisson clocks. We can see in the following construction that the size (up to time $\tau$) is upper bounded by a jump process of rate $4dm$, with random jumps of size $m_i \indicator{U_i<m_i^{-\alpha}}$ where $m_i$ are also random, and $(U_i)_{i \in \BN}$ are independent copies of $U$. We construct the process in the following way:

First sample $U_{i,k} \sim \operatorname{Unif}[0,1]$ and $t_{i,k} \sim \textup{Exp}(1)$ for $i\in \BN$ and $k=0,1,2,...,4dm-1$. Run the Cluster-Cluster model while using $t_{i,k},U_{i,k}$ as the clock and generator of the $k^{\text{th}}$ neighboring cluster, with $k=0$ referring instead to $\mathfs{C}_0$. Note that the neighboring clusters may change, and that at any given time we consider the clusters which are currently adjacent. Also note that we may have more clocks than neighboring clusters, in which case we consider "ghost clusters of size 0" which do not have any effect. Now using the Cluster-Cluster model, we take $m_{i,k}$ to be the size of the cluster which we gave clock $t_{i,k}$ to at the time the clock rang. Note that we always take $m_{i,0} = m$.

Note that for each $k$, $t_{i,k},m_{i,k},U_{i,k}$ we get a jump process as described above with rate 1. Combining the $4dm$ processes gives a single jump process. For this process, we define $t_i,U_i,m_i$ accordingly using $t_{i,k},m_{i,k},U_{i,k}$. Note that we created this process via coupling to the Cluster-Cluster model. We will denote the law of this coupling $\Prob[C]{\cdot}$, and we will also denote the expectation under it by $\E[C]{\cdot}$. We will also denote the jump process by $N(t) = N([0,t])$, and the number of arrivals using $\Tilde{N}(t) = \Tilde{N}([0,t])$, so that $N(t) = \sum_{i=1}^{\Tilde{N}(t)}{m_i\indicator{U_i<m_i^{-\alpha}}}$. Note the $\Tilde{N}$ is a Poisson Process with jumps of size 1 and rate 1.
We can see that under this coupling, the doubling time in the Cluster-Cluster model happened after the Poisson process reached size $m$. 
Note that $m_i$ is independent of $U_i$ even given past arrival times $t_j$, $j\le i$, due to Markov property. 
Note that since $m_i$ and $U_i$ are independent of future arrival times $t_j$ for $j>i$, so overall we get that $m_i$ is independent of $U_i$ given all arrival times $t_j$. We get that
\begin{equation*}
    \begin{aligned}
        \Prob{\Delta \mathcal{T}_n \le \delta} \le \Prob[C]{N(\delta)\ge m}&= \sum_{k=1}^{\infty}\Prob[C]{\Tilde{N}(\delta)=k,\sum_{i=1}^{k}{m_i\indicator{U_i<m_i^{-\alpha}}}\ge m}\\
        &= \sum_{k=1}^{\infty}\Prob[C]{\Tilde{N}(\delta)=k} \Prob{\sum_{i=1}^{k}{m_i\indicator{U_i<m_i^{-\alpha}}}\ge m \Bigg| \Tilde{N}(\delta)=k}.\\
    \end{aligned}
\end{equation*}
Calculating the expectation of $\sum_{i=1}^{k}{m_i\indicator{U_i<m_i^{-\alpha}}}$ conditioned on $\Tilde{N}(\delta)=k$
\begin{equation}\label{sumBound2}
    \begin{aligned}
        &\E[C]{\sum_{i=1}^{k}{m_i\indicator{U_i<m_i^{-\alpha}}} \Bigg| \Tilde{N}(\delta)=k} = \sum_{i=1}^{k}\E[C]{{m_i\indicator{U_i<m_i^{-\alpha}}} \Bigg| \Tilde{N}(\delta)=k} \\
        &\le \sum_{i=1}^{k}\sum_{v=1}^{\infty}\E[C]{{v\indicator{U_i<v^{-\alpha}}} \Bigg| \Tilde{N}(\delta)=k}\Prob[C]{m_i = v | \Tilde{N}(\delta)=k} \\
        &= \sum_{i=1}^{k}\sum_{v=1}^{\infty} v^{1-\alpha} \Prob[C]{m_i = v | \Tilde{N}(\delta)=k}  \le \sum_{i=1}^{k}\sum_{v=1}^{\infty} \Prob[C]{m_i = v | \Tilde{N}(\delta)=k} \leq k,
    \end{aligned}
\end{equation}
where in the first inequality we use the law of total probability, the fact that if $v=\infty,0$ then $v\indicator{U_i<v^{-\alpha}} = 0$ almost surely, and the conditional independence of $m_i,U_i$. In the last inequality we use the fact that $\alpha \ge 1$ and $v\ge 1$. Using Markov's inequality we get
\begin{equation}\label{markovInequality}
    \begin{aligned}
        \Prob[C]{\sum_{i=1}^{k}{m_i\indicator{U_i<m_i^{-\alpha}}}\ge m \Bigg| \Tilde{N}(\delta)=k} \le \frac{k}{m} . 
    \end{aligned}
\end{equation} 
Continuing \eqref{sumBound2} we get for any choice of $c$ such that $\frac{m}{c}\in \mathds{Z}$ and $c>4$: 
\begin{equation}\label{sumBound_3}
    \begin{aligned}
        \Prob{\Delta \mathcal{T}_n \le \delta} &\le \sum_{k=1}^{\infty}\Prob[C]{\Tilde{N}(\delta)=k} \min\rb{1,\frac{k}{m}} \\
        &\le \sum_{k=1}^{\frac{m}{c}}\Prob[C]{\Tilde{N}(\delta)=k} \frac{k}{m} + \sum_{\frac{m}{c}+1}^{\infty}\Prob[C]{\Tilde{N}(\delta)=k}\\
        &\le \sum_{k=1}^{\frac{m}{c}}\Prob[C]{\Tilde{N}(\delta)=k} \frac{\frac{m}{c}}{m} + \sum_{\frac{m}{c}+1}^{\infty}\Prob[C]{\Tilde{N}(\delta)=k} \\
        &\le \frac{1}{c} + \Prob[C]{\Tilde{N}(\delta) \ge \frac{m}{c}+1} \\
        &\le \Prob[C]{\operatorname{Pois}(4dm\delta) \ge \frac{m}{c}} + \frac{1}{c} \\
        &\le e^{-m\delta_d h(\frac{1}{\delta_d c}-1)} + \frac{1}{c}
    \end{aligned}
\end{equation}
Where $\delta_d = 4d\delta$, we use Bennett's inequality, requiring $\frac{1}{\delta_d c} \ge 2$ and $h(u) = (1+u)\log\rb{1+u}-u$. As we have previously seen, we also require $c>4$. We can also assume $m$ is large enough, since after $n$ doubling, the size is at least $2^n$. Testing numerically, we can see that we can choose appropriate $\delta$ and $c$.
\end{proof}
Using Lemma \ref{lem:infinite doubling}, we are now ready to prove the main result of this section. 
The main idea in the proof of Theorem \ref{thm:finite_clusters} is that since a cluster must have an infinite neighbor to become infinite, we see that for a cluster of size $m$ to become infinite in the time window $[t,t+\ep]$, it must move, and at least one of the (at most) $2dm$ clusters adjacent to it must already be infinite. 
If we denote the probability for a cluster being infinite at time $t$ by $f(t)$ and use a union bound, we get $2dmf(t)$ for the probability for an infinite neighbor, and $\frac{\ep}{m}$ for the probability of the cluster moving. 
Together this gives $2df(t)\ep$, and after we take $\ep \to 0$, we will get a differential equation with solution $f(t) \le e^{ct}$. Since requiring a cluster to move in the time interval $[0,t]$ gives $f(t)\le t$, we will get that $f(t)$ must be 0.


\begin{proof}[Proof of Theorem \ref{thm:finite_clusters}, case $\alpha\ge 1$]
Denote by $K_x(t)$ the cluster that contains the vertex $x$ at time $t$. Note that this may refer to different clusters at different times. 
Consider $\varkappa = \inf\cb{t' \ge 0 :\Prob{K_0(t') = \infty}>0}$.
Assume by contradiction that $\varkappa<\infty$, and define $f(s) = \Prob{K_0(\varkappa+s) = \infty}$.
Let $\ep>0$. By Lemma~\ref{lem:infinite doubling}, for a vertex $x$, to first join an infinite cluster in the time interval $[\varkappa+s,\varkappa+s+\ep]$, $K_x(\cdot)$ must connect to an already infinite cluster, by moving in this time interval.
Note that if we apply this to the time interval $[\varkappa - \ep, \varkappa + s]$, by definition of $\varkappa$, gives us that $f(s) \le s+\ep$ for all $s\ge 0$. Taking $\ep \to 0$ we get $f(s) \le s$. 
Denote by $T_z(\nu,\eta)$ the amount of times that the clock of cluster $K_z$ rang in the time interval $(\varkappa+\nu,\varkappa+\eta]$.
We emphasize that this does not count the amount of times a clock of a unique cluster rings, but instead it counts the amount of times there was a cluster containing vertex $z$ at some time $t\in(\varkappa+\nu,\varkappa+\eta]$ whose clock rang at time $t$.
We also denote by $I_z(\nu,\eta) = \cb{|K_z(\varkappa+\eta)| = \infty, |K_z(\varkappa+\nu)|<\infty}$ the event that the cluster $K_z$ became infinite in the time interval $(\varkappa+\nu,\varkappa+\eta]$. This gives
\begin{equation}\label{eq:moivein}
    \begin{aligned}
            f(s+\ep) - f(s) = \Prob{I_0(s,s+\ep)} \le \Prob{I_0(s,s+\ep),T_0(s,s+\ep) = 1} + O(\ep^2)
    \end{aligned}
\end{equation}
since the cluster must move to become infinite, and since the probability of moving twice is at most $\ep^2$. 
Note that if $K_0=\emptyset$, then a neighboring cluster must move twice which happens with probability at most $2d\ep^2$. 
Under the event $\cb{I_0(s,s+\ep),T_0(s,s+\ep) = 1}$, since 0 moves only once in interval $(\varkappa+s,\varkappa+s+\ep]$, there must be an infinite cluster at time $\varkappa+s$ at its outer boundary $\partial K_0(\varkappa+s)$, since otherwise the cluster would be surrounded by non-infinite clusters, which implies that it is also not infinite. Continuing \eqref{eq:moivein}:
\begin{equation}\label{eq:onemove}
    \begin{aligned}
            f(s+\ep) - f(s) &= \Prob{I_0(s,s+\ep),T_0(s,s+\ep) = 1,\cup_{x \in \partial K_0(\varkappa+s)}{I_x(0,s)}} \\
            &+ \Prob{I_0(s,s+\ep),T_0(s,s+\ep) = 1,\cup_{x \in \partial K_0(\varkappa+s)}{I_x(s,s+\ep)}} + O(\ep^2) \\
            &\le \Prob{T_0(s,s+\ep) = 1,\cup_{x \in \partial K_0(\varkappa+s)}{I_x(0,s)}} \\
            &+ \Prob{T_0(s,s+\ep) = 1,\cup_{x \in \partial K_0(\varkappa+s)}{T_x(s,s+\ep) \ge 1}} + O(\ep^2). 
        \end{aligned}
\end{equation}
Now conditioning on $K_0(\varkappa+s)$, denoting by $\mathcal{S}_m\subset (\BZ^d)^m$ the set of connected clusters of size $m$. Continuing \eqref{eq:onemove}, we see that $f(s+\ep) - f(s)$ is bounded from above by 
\begin{equation*}
    \begin{aligned}
           &= \sum_{m=1}^{\infty}\sum_{\mathfs{C} \in \mathcal{S}_m}\Prob{T_0(s,s+\ep) = 1,\cup_{x \in \partial \mathfs{C}}{I_x(0,s)}\Big|K_0(\varkappa+s) = \mathfs{C}}\Prob{K_0(\varkappa+s) = \mathfs{C}} \\
            &+ \sum_{m=1}^{\infty}\sum_{\mathfs{C} \in \mathcal{S}_m}\Prob{T_0(s,s+\ep) = 1,\cup_{x \in \partial \mathfs{C}}{T_x(s,s+\ep) \ge 1}\Big|K_0(\varkappa+s) = \mathfs{C}}\Prob{K_0(\varkappa+s) = \mathfs{C}} + O(\ep^2) \\
            &= \sum_{m=1}^{\infty}\sum_{\mathfs{C} \in \mathcal{S}_m}\Prob{T_0(s,s+\ep) = 1\Big| \cup_{x \in \partial \mathfs{C}}{I_x(0,s)},K_0(\varkappa+s) = \mathfs{C}} \\ 
            & \qquad \quad \cdot \Prob{\cup_{x \in \partial \mathfs{C}}{I_x(0,s)}\Big|K_0(\varkappa+s) = \mathfs{C}}\Prob{K_0(\varkappa+s) = \mathfs{C}} \\
            &+ \sum_{m=1}^{\infty}\sum_{\mathfs{C} \in \mathcal{S}_m}\Prob{T_0(s,s+\ep) = 1\Big| \cup_{x \in \partial \mathfs{C}}{T_x(s,s+\ep) \ge 1},K_0(\varkappa+s) = \mathfs{C}} \\
            & \qquad \quad \cdot \Prob{\cup_{x \in \partial \mathfs{C}}{T_x(s,s+\ep) \ge 1}\Big|K_0(\varkappa+s) = \mathfs{C}}\Prob{K_0(\varkappa+s) = \mathfs{C}} + O(\ep^2) \\
            &\le \sum_{m=1}^{\infty}\frac{\ep}{m}\sum_{\mathfs{C} \in \mathcal{S}_m}\sum_{x \in \partial \mathfs{C}}\Prob{{I_x(0,s)}\Big|K_0(\varkappa+s) = \mathfs{C}}\Prob{K_0(\varkappa+s) = \mathfs{C}} \\
            &+ \sum_{m=1}^{\infty}\frac{\ep}{m}\sum_{\mathfs{C} \in \mathcal{S}_m}\sum_{x \in \partial \mathfs{C}} \Prob{{T_x(s,s+\ep) \ge 1}\Big|K_0(\varkappa+s) = \mathfs{C}}\Prob{K_0(\varkappa+s) = \mathfs{C}} + O(\ep^2) \\ 
            &\le \sum_{m=1}^{\infty}\frac{\ep}{m}\sum_{\mathfs{C} \in \mathcal{S}_m}\sum_{x \in \partial \mathfs{C}}\Prob{{I_x(0,s)},K_0(\varkappa+s) = \mathfs{C}} + \sum_{m=1}^{\infty}\frac{\ep}{m}\sum_{\mathfs{C} \in \mathcal{S}_m}2dm\ep\Prob{K_0(\varkappa+s) = \mathfs{C}} + O(\ep^2) 
    \end{aligned}
\end{equation*}
where in the last inequality, we use the independence of clocks of different clusters, $\alpha \ge 1$, and the independence of the clock at time interval $(\varkappa+s,\varkappa+s+\ep)$ to what happened in previous times (by Markov property). 
We can now focus on the first term and use translation invariance (where $\mathfs{C}-x$ denotes translation by $x$). 
Let $E_t(\mathfs{C})$ be the event that the cluster $\mathfs{C}$ is a cluster in the model at time $t$. Note that $\forall x\in \BZ^d: \cb{K_x(t) = \mathfs{C}} = E_t(\mathfs{C})\indicator{x\in \mathfs{C}}$. This gives: 
\begin{equation}
    \begin{aligned}
    &\sum_{m=1}^{\infty}\frac{\ep}{m}\sum_{\mathfs{C} \in \mathcal{S}_m}\sum_{x \in \partial \mathfs{C}}\Prob{{I_x(0,s)},K_0(\varkappa+s) = \mathfs{C}} \\
            &= \sum_{m=1}^{\infty}\frac{\ep}{m}\sum_{\mathfs{C} \in \mathcal{S}_m}\sum_{x \in \partial \mathfs{C}}\Prob{{I_0(0,s)}, K_{-x}(\varkappa+s) = \mathfs{C}-x} \\
            &= \sum_{m=1}^{\infty}\frac{\ep}{m}\sum_{\mathfs{C} \in \mathcal{S}_m}\sum_{x \in \mathds{Z}^d}\Prob{I_0(0,s), E_{\varkappa+s}(\mathfs{C}-x)}\indicator{0\in \mathfs{C}}\indicator{x \in \partial \mathfs{C}}\\
            &= \sum_{m=1}^{\infty}\frac{\ep}{m}\sum_{x \in \mathds{Z}^d}\sum_{\mathfs{C} \in \mathcal{S}_m}\Prob{I_0(0,s), E_{\varkappa+s}(\mathfs{C}-x)}\indicator{-x\in \mathfs{C}-x}\indicator{0 \in \partial (\mathfs{C}-x)}\\
    \end{aligned}
\end{equation}
Now note that the sum over $\mathfs{C}\in \mathcal{S}_m$ is the same as the sum over $(\mathfs{C}-x) \in \mathcal{S}_m$: 
\begin{equation}\label{eq:SecondToLast}
    \begin{aligned}
            &= \sum_{m=1}^{\infty}\frac{\ep}{m}\sum_{x \in \mathds{Z}^d}\sum_{\mathfs{C} \in \mathcal{S}_m}\Prob{I_0(0,s), E_{\varkappa+s}(\mathfs{C})}\indicator{-x\in \mathfs{C}}\indicator{0 \in \partial \mathfs{C}} \\
            &= \sum_{m=1}^{\infty}\frac{\ep}{m}\sum_{\mathfs{C} \in \mathcal{S}_m}\Prob{I_0(0,s), E_{\varkappa+s}(\mathfs{C})}\indicator{0 \in \partial \mathfs{C}}\sum_{x \in \mathds{Z}^d}\indicator{-x\in \mathfs{C}} \\
            &= \sum_{m=1}^{\infty}\frac{\ep}{m}\sum_{\mathfs{C} \in \mathcal{S}_m}\Prob{I_0(0,s), E_{\varkappa+s}(\mathfs{C})}\indicator{0 \in \partial \mathfs{C}}m \\
            &\le \sum_{m=1}^{\infty}\frac{\ep m}{m}\sum_{\mathfs{C} \in \mathcal{S}_m}\Prob{I_0(0,s), E_{\varkappa+s}(\mathfs{C})}\sum_{i=1}^{2d}{\indicator{e_i \in \mathfs{C}}} \\
            &= \sum_{m=1}^{\infty}\ep\sum_{i=1}^{2d}\sum_{\mathfs{C} \in \mathcal{S}_m}\Prob{I_0(0,s), E_{\varkappa+s}(\mathfs{C})}{\indicator{e_i \in \mathfs{C}}} \\
    \end{aligned}
\end{equation}
Now using rotational symmetry, continuing \eqref{eq:SecondToLast}:

\begin{equation}
    \begin{aligned}
            &= \sum_{m=1}^{\infty}2d\ep\sum_{\mathfs{C} \in \mathcal{S}_m}\Prob{I_0(0,s), E_{\varkappa+s}(\mathfs{C})}\indicator{e_1 \in \mathfs{C}} \\
            &= 2d\ep\sum_{m=1}^{\infty}\sum_{\mathfs{C} \in \mathcal{S}_m}\Prob{I_0(0,s), E_{\varkappa+s}(\mathfs{C})}\indicator{e_1 \in \mathfs{C}} \\
            &\le 2d\ep\Prob{I_0(0,s)} = 2d\ep f(s),
    \end{aligned}
\end{equation}
 where the events are disjoint since we sum on different $\mathfs{C}$ and they all contain $e_1$.
Dividing by $\ep$ and taking the limit as $\ep \to 0$ we get,
\begin{equation}
    \begin{aligned}
        f'(s) \le 2df(s)
    \end{aligned}
.\end{equation}
Applying Gr\"onwall's inequality, starting at $\ep$, since $f$ is monotone (trivially by coupling via the process) we get the inequality
\begin{equation}
    \begin{aligned}
        f(x) \le f(x+\ep) \le f(\ep)e^{2dx},
    \end{aligned}
\end{equation}
and if we take $\ep \to 0$, since $f(\ep) \le \ep$, we get that $f(x) = 0$, which is a contradiction.
\end{proof}

\end{document}